\documentclass[11pt]{amsart}  
\usepackage{amsthm, amsmath, amscd, amssymb, latexsym, stmaryrd, color}
\usepackage[top=2.5cm, bottom=2.5cm, left=2cm, right=2cm]{geometry}
\usepackage[colorinlistoftodos]{todonotes}
\theoremstyle{plain}

\newtheorem*{theorem*}{Theorem}
\newtheorem{theorem}{Theorem}[section]
\newtheorem{lem}[theorem]{Lemma}
\newtheorem{prop}[theorem]{Proposition}

\newtheorem{cor}[theorem]{Corollary}
\theoremstyle{definition}

\newtheorem{defn}[theorem]{Definition}

\newtheorem{ex}[theorem]{Example}

\newtheorem{rmk}[theorem]{Remark}

\theoremstyle{remark}

\numberwithin{equation}{section}
\usepackage[mathscr]{eucal}
\usepackage{stackrel,amssymb}
\usepackage{graphics, graphpap}
\usepackage{array, tabularx, longtable}
\usepackage{color}

\usepackage{url}
\usepackage[T1]{fontenc}
\usepackage{array,epsfig}
\usepackage{amsmath}
\usepackage{old-arrows}
\usepackage{amsfonts}
\usepackage{amssymb}
\usepackage{amsxtra}
\usepackage{latexsym}
\usepackage{dsfont}
\usepackage[thinlines]{easytable}
\usepackage{makecell}
\usepackage{tabularray}
\usepackage{mathrsfs}
\usepackage{color}
\usepackage[all]{xy}
\usepackage{}
\usepackage{xfrac}
\usepackage{tikz}
\usepackage{tikz-cd}
\usepackage{graphicx}
\usepackage{mathtools}
\usepackage{caption}
\makeatletter
\newcommand\mathcircled[1]{%
  \mathpalette\@mathcircled{#1}%
}
\newcommand\@mathcircled[2]{%
  \tikz[baseline=(math.base)] \node[draw,circle,inner sep=1pt,color=red] (math) {$\m@th#1#2$};%
}
\makeatother

\usepackage[pagebackref=true]{hyperref}

\usepackage[OT2,T1]{fontenc}
\usepackage[page,toc,titletoc,title]{appendix}

\hypersetup{
 colorlinks,
 linkcolor=blue, citecolor=blue
 }
 
\newcommand{\mscr}{\mathscr{M}}

\newcommand{\rscr}{\mathscr{R}}

\newcommand{\uscr}{\mathscr{U}}

\newcommand{\acal}{\mathcal{A}}

\newcommand{\jcal}{\mathcal{J}}
\newcommand{\kcal}{\mathcal{K}}

\newcommand{\tcal}{\mathcal{T}}

\newcommand{\xfrak}{\mathfrak{X}}

\newcommand{\anor}{\textnormal{A}}

\newcommand{\hbf}{\mathbf{H}}

\newcommand{\lnor}{\textnormal{L}}
\newcommand{\mnor}{\textnormal{M}}

\newcommand{\et}{\textnormal{\'et}}

\newcommand{\an}{\textnormal{an}}

\newcommand{\loc}{\textnormal{loc}}
\newcommand{\mot}{\textnormal{mot}}

\newcommand{\ct}{\textnormal{ct}}

\newcommand{\univ}{\textnormal{univ}}
\newcommand{\geo}{\textnormal{geo}}
\newcommand{\un}{\textnormal{un}}

\newcommand{\id}{\operatorname{id}}

\newcommand{\pr}{\operatorname{pr}}
\newcommand{\Hom}{\operatorname{Hom}}

\renewcommand{\sp}{\operatorname{sp}}
\newcommand{\daet}{\mathbf{DA}^{\et}}
\newcommand{\daetct}{\mathbf{DA}^{\et}_{\textnormal{ct}}}
\newcommand{\dcat}{\mathbf{D}}
\newcommand{\dcatct}{\mathbf{D}_{\textnormal{ct}}}
\newcommand{\dn}{\mathbf{DN}}

\newcommand{\betti}{\operatorname{Bti}}
\newcommand{\nori}{\textnormal{Nori}}

\renewcommand{\pr}{\operatorname{pr}}

\newcommand{\gm}{\mathbb{G}}
\newcommand{\phnor}{{}^p\mathrm{H}}

\newcommand{\rat}{\operatorname{rat}}

\newcommand{\var}{\operatorname{Var}}

\newcommand{\perv}{\operatorname{Perv}}

\newcommand{\localsystem}{\operatorname{Loc}}

\newcommand{\op}{\operatorname{op}}

\newcommand{\colim}{\operatorname{colim}}

\newcommand{\Spec}{\operatorname{Spec}}

\DeclareSymbolFont{cyrletters}{OT2}{wncyr}{m}{n}
\DeclareMathSymbol{\Sha}{\mathalpha}{cyrletters}{"58}
\DeclareMathSymbol{\Be}{\mathalpha}{cyrletters}{"42}

\title[Motivic nearby functors on perverse Nori motives]{Motivic nearby functors on perverse Nori motives}  

\author[Khoa Bang. P]{Khoa Bang Pham}
\address{The Hong Kong Universtiy of Science and Technology, Clear Water Bay, Hong Kong}
\email{phamkb@ust.hk}
\thanks{}

\begin{document}           
\begin{abstract}
In this article, we show that several candidate motivic (unipotent) nearby functors coincide on perverse Nori motives. In particular, the canonical functor defined via the universal abelian factorization admits an expression in terms of the six operations and yields a monodromy sequence on perverse Nori motives. We then deduce the motivic integral identity of Kontsevich-Soibelman for perverse Nori motives. In the appendix, we also prove a version of Beilinson's equivalence for perverse sheaves of geometric origin, which is used repeatedly throughout the article.
\end{abstract}
\maketitle                 


\section{Introduction}

\subsection{State of the art} In \cite{braden-2003}, Braden proves the so-called hyperbolic localization theorem, which is later used in many places; for instance, in the proof of the geometric Satake equivalence of Mirković–Vilonen \cite{mirkovic+vilonen-2007}. Roughly speaking, let $X$ be a  scheme equipped with a sufficiently good action (to be made precise later) of the multiplicative group $\gm_m$. One can then form the \textit{space of fixed points}, the \textit{attractor} and the \textit{repeller} $X^{\circ},X^+,X^-$, respectively, 
\begin{align*}
    X^0 & = \left \{x \in X \mid gx = x \ \forall \ g \in \gm_m \right \} \\ 
    X^+ & = \left \{x \in X \mid \exists \lim_{g \to 0 } gx \right \} \\
    X^- & =\left \{x \in X \mid \exists \lim_{g \to \infty} gx \right \}.
\end{align*}
There are natural morphisms: two inclusions $e^{\pm} \colon X^{\circ} \longrightarrow X^{\pm}$ and two limit morphisms $\pi^{\pm} \colon X^{\pm} \longrightarrow X^{\circ}$ (sending $x$ to $\lim_{g \to 0}gx$ and $\lim_{g \to \infty}gx$). The Braden transformation is a natural transformation
\begin{equation*}
    (\pi^-)_*(e^-)^!(M) \longrightarrow (\pi^+)_!(e^+)^*(M)
\end{equation*}
The Braden transformation is an isomorphism if $M$ is $\gm_m$-equivariant, namely, $a^*(M) \simeq p^*(M)$, where $a,p \colon \gm_m \times X \longrightarrow X$ is the action morphism and the projection, respectively. Several subsequent works have developed functorial approaches directly related to \cite{braden-2003}, the reader can consult \cite{richarz-2018}\cite{drinfeld+gaitsgory-2014}\cite{drinfeld-2015}. The commutation between the Braden transformations and nearby functors plays an important role in different areas. In geometric representation theory, it can be viewed as a functorial enhancement of the compatibility between the Bernstein isomorphism and the constant terms maps (see \cite{richarz-2021}). On the other hand, in Donaldson-Thomas theory, this yields a motivic enhancement of the integral identity of Kontsevich-Soibelman (see \cite{kontsevich-2008}) thanks to the works \cite{bang-2024} (see also \cite{florian-2024}). Since Braden's theorem holds whenever one has a six-functor formalism, this poses the following natural question: when do Braden transformations commute with nearby functors? As showed in \cite{bang-2024} (see \cite{richarz-2018} for the \'etale setting), the answer is positive whenever the nearby functor is expressed in terms of six operations. However, the story is unclear for perverse Nori motives. Before entering the story, let us say a few words about perverse Nori motives. Let $X$ be a scheme and let $\daet(X)$ be the category of \'etale motives with rational coefficients (see \cite{ayoub-2014}). Let $\daetct(X)$ be the category of constructible \'etale motives. By the work \cite{ayoub-2010}, there is a Betti realization functor 
  \begin{equation*}
      \betti^*_X \colon \daetct(X) \longrightarrow \dcatct^b(X)
  \end{equation*}
  where $\dcatct^b(X)$ is the bounded derived category with algebraically constructible cohomology. In \cite{florian+morel-2019}, Ivorra, Morel defines perverse Nori motives $\mscr\perv(X)$ as the universal abelian factorization \footnote{Since the image of the Betti realization lands in complexes of geometric origin, in this definition, from now we can replace $\dcatct^b(X),\perv(X)$ by $\dcat^b_{\geo}(X),\perv_{\geo}(X)$, those of geometric origin.}
\begin{equation*}
    \begin{tikzcd}[sep=large]
              \daetct(X) \arrow[r,"\betti^*_X"] \arrow[d,"\phnor^0_{\univ}",swap] & \dcat_{\ct}^b(X) \arrow[r,"\phnor^0"] & \perv(X) \\
              \mscr\perv(X) \arrow[rru,"\rat_X",swap] &  & 
    \end{tikzcd}
\end{equation*}
In some sense, $\mscr\perv(X)$ is expected to be the heart of certain perverse motivic $t$-structure on \'etale motives $\daetct(X)$. In \cite{tubach-2025}, Tubach proves that this is indeed the case under the standard conjectures, namely, $\daetct(X) \simeq \dn^b(X)$. Thus, one can imagine that there should exist a six-functor formalism, a theory of motivic nearby functors for perverse Nori motives. The six-functor formalism is realized in \cite{florian+morel-2019} and \cite{terenzi-2025} but the developement of nearby functors is not yet established therein. Let $f \colon X\longrightarrow \mathbb{A}_k^1$ be a morphism of $k$-varieties, let us consider the following diagram with cartesian squares
\begin{equation*}
    \begin{tikzcd}[sep=large]
       X _{\eta} \arrow[r,"j_f"]  \arrow[d,"f_{\eta}",swap] & X  \arrow[d,"f"] & X_{\sigma} \arrow[d,"f_{\sigma}"] \arrow[l,"i_f",swap] \\ 
        \eta \coloneqq \mathbb{G}_{m,k} \arrow[r,"j_{\id}"] & \mathbb{A}_k^1 & \sigma \coloneqq \Spec(k). \arrow[l,"i_{\id}",swap],
    \end{tikzcd}
    \end{equation*} 
where $i \colon \Spec(k) \longhookrightarrow \mathbb{A}_k^1$ is the zero section and $j \colon \gm_{m,k} \longhookrightarrow \mathbb{A}_k^1$ is its open complement. There is an obvious candidate for motivic (unipotent) nearby functors associated with $f$ (on perverse Nori motives)
\begin{equation*}
    \Psi_f^{\univ}[-1],\Upsilon_f^{\univ}[-1] \colon \mscr\perv(X_{\eta}) \longrightarrow \mscr\perv(X_{\sigma}),
\end{equation*} 
defined in terms of the universal property since the usual (unipotent) nearby functors 
\begin{equation*} 
\Psi_f^{\an}[-1],\Upsilon_f^{\an}[-1] \colon \dcatct^b(X_{\eta}) \longrightarrow \dcatct^b(X_{\sigma})
 \end{equation*}
 are perverse $t$-exact. However, this definition lacks functorialities in the sense that it does not come from six operations, which is one of the key step in the proof of the motivic integral identity in \cite{florian-2024}\cite{bang-2024} and moreover, it has poor $\infty$-categorial features since it is only defined on abelian categories. Therefore, we expect $\Psi_f^{\univ}$ admits an expression in terms of six operations. A similar problem arises for monodromy operators 
 \begin{equation*}
     N^{\univ} \colon \Upsilon_f^{\univ}[-1] \longrightarrow \Upsilon_f^{\univ}[-1](-1),
 \end{equation*}
defined by the universal property on abelian categories perverse Nori motives and it is natural to wonder whether it gives rise to a monodromy sequence on derived categories
\begin{equation*}
    i^*j_* \longrightarrow \Upsilon_f^{\univ} \longrightarrow \Upsilon_f^{\univ}(-1) \longrightarrow +1
\end{equation*} 
In this article, we would like to provide affirmative answers to all problems addressed above. We emphasize that our approach does not rely on the standard conjectures. In the course of this work, we also develop several auxiliary results concerning nearby functors on Nori motives and constructible complexes of geometric origin. 
\subsection{Statements of the main results}
Throughout this paper, $k$ is a subfield of $\mathbb{C}$. Let $X$ be a $k$-variety, we denote by $\dn^b(X) = \dcat^b(\mscr\perv(X))$ the unbounded derived category of perverse Nori motives $\mscr\perv(X)$ and $\dn(X) = \operatorname{Ind}(\dn^b(X))$ its $\infty$-ind completion. In the subsequent, functors with superscript $(-)^{\anor}$ mean those constructed by Ayoub, functors with superscript $(-)^{\univ}$ mean those obtained by the universal property of perverse Nori motives and functors with superscript $(-)^{\an}$ mean the usual ones. 
\begin{theorem} \label{thm: first main theorem}
    Let $f \colon X \longrightarrow \mathbb{A}_k^1$ be a morphism of $k$-varieties, there exist natural isomorphisms of functors
    \begin{align*}
        \Psi_f^{\anor} \simeq \Psi_f^{\univ} \colon \dn^b(X_{\eta}) \longrightarrow \dn^b(X_{\sigma}) \\ 
       \log_f^{\nori} \simeq \Upsilon_f^{\anor} \simeq \Upsilon_f^{\univ} \colon \dn^b(X_{\eta}) \longrightarrow \dn^b(X_{\sigma})
    \end{align*}
    where $\Psi_f^{\anor}, \Upsilon_f^{\anor},\log_f^{\nori}$ are Ayoub's constructions of the motivic nearby functor, the motivic unipotent nearby functor, the logarithmic specialization system, respectively (see \cite{ayoub-thesis-2}\cite{ayoub-nearby-cycles}\cite{ayoub-2014}). Consequently, there is a Nori-monodromy sequence
    \begin{equation*}
       (\textnormal{Mon}^{\univ}) \colon i^*j_* \longrightarrow \Upsilon_f^{\univ} \overset{N^{\univ}}{\longrightarrow} \Upsilon_f^{\univ}(-1) \longrightarrow +1
    \end{equation*}
    on $\dn(X_{\sigma})$ and $N^{\univ}$ is nilpotent on $\dn^b(X_{\sigma})$. If we denote by $\textnormal{Mon}^{\anor},\textnormal{Mon}^{\an}$ the corresponding monodromy sequences of the \'etale-motivic setting and the classical Betti setting, then the result can be visualized as follows
    \begin{equation*}
    \begin{tikzcd}[column sep=large, row sep = 0.2]
        \daetct(-) \arrow[rr,"\betti^*",bend left =  30] \arrow[r,"\operatorname{Nri}^*"] & \dn^b(-) \arrow[r] & \dcatct^b(-) \\ 
       \Psi_f^{\anor},\Upsilon_f^{\anor} \arrow[r] & \Psi_f^{\univ}, \Upsilon_f^{\univ} \arrow[r] & \Psi_f^{\an},\Upsilon_f^{\an} \\ 
          \textnormal{Mon}^{\anor} \arrow[r] & \textnormal{Mon}^{\univ}\arrow[r] & \textnormal{Mon}^{\an}
        \end{tikzcd}
    \end{equation*}
\end{theorem}
Let us unravel several hidden aspects of the theorem that cannot be overestimated. First, the Ayoub's functors are first defined on ind-Nori motives $\dn(X_{\eta}) = \operatorname{Ind}(\dn^b(X_{\eta}))$ and the isomorphisms $\Upsilon_f^{\anor} \simeq \Upsilon_f^{\univ}, \Psi_f^{\anor} \simeq \Psi_f^{\univ}$ implicitly say that Ayoub's functors preserve (bounded) Nori motives. Second, the existence of the Nori-monodromy operator does not come for free. Since one expects to define it via the universal property, one has to show that the image of the motivic monodromy sequence under the Betti realization is the constructible monodromy sequence and after that, $N^{\nori}$ is only defined on abelian categories and hence computing its homotopy fiber (which is $i^*j_*$) is highly nontrivial. Third, the compatibility of monodromy sequences (and also unipotent nearby functors) under the Betti realization is a "unipotent" analogue of \cite{ayoub-2010}; however, in \cite{ayoub-2010}, the compatibility is at the level of functors, not at the level of sequences. In order to solve these problems, here we have to compare unipotent functors with logarithmic specialization systems on constructible complexes of geometric origin (for the motivic setting, see \cite{ayoub-thesis-2}\cite{ayoub-nearby-cycles}\cite{ayoub-2014}). Lastly, we deduce the Kontsevich-Soibelman integral identity on Nori motives
\begin{cor}
Let $X$ be a $k$-variety equipped with an action of $\gm_{m,k}$ and let $f \colon X \longrightarrow \mathbb{A}_k^1$ be a $\gm_m$-equivariant morphism of $k$-varieties. There is a commutative diagram of the form
    \begin{equation*}
        \begin{tikzcd}[sep=large]
            (\pi^-_{\sigma})_*(e^-_{\sigma})^!\Psi_f^{\univ}(M) \arrow[d] & \Psi_{f^0}^{\univ}(\pi^-_{\eta})_*(e^-_{\eta})^!(M) \arrow[d] \arrow[l] \\ 
            (\pi^+_{\sigma})_!(e^+_{\sigma})^*\Psi_f^{\univ}(M) \arrow[r]  & \Psi_{f^0}^{\univ} (\pi_{\eta}^+)_!(e^+_{\eta})^*(M)
        \end{tikzcd}
    \end{equation*}
natural in $M \in \dn^b(X_{\eta})$. Moreover, then all arrows are isomorphisms provided that $M$ is $\gm_m$-equivrariant. In particular, let $X$ be a $k$-variety and let $\mathbb{G}_{m,k}$ act on $\mathbb{A}_k^{d_1} \times_k \mathbb{A}_k^{d_2} \times_k X$ by 
$\lambda(\mathbf{a}_1,\mathbf{a_2},x) = (\lambda^{>0}\mathbf{a}_1,\lambda^{<0}\mathbf{a}_2,x)  \ \forall \ \lambda \neq 0$. Let $f \colon \mathbb{A}_k^{d_1} \times_k \mathbb{A}_k^{d_2} \times_k X \longrightarrow \mathbb{A}_k^1$ be a $\mathbb{G}_{m,k}$-equivariant morphism, where $\mathbb{A}_k^1$ is endowed with the trivial $\mathbb{G}_{m,k}$-action, i.e., on points,
\begin{equation*}
    f(\mathbf{a}_1,\mathbf{a_2},x) = f(\lambda^{>0}\mathbf{a}_1,\lambda^{<0}\mathbf{a}_2,x).
\end{equation*}
Let $X$ be embedded in $\mathbb{A}_k^{d_1} \times_k \mathbb{A}_k^{d_2} \times_k X$ by zero sections, then there is an isomorphism of motives
    \begin{equation*}
        (\pi^+_{\sigma})_!(e^+_{\sigma})^*\Psi_f^{\univ}(\mathds{1}_{X_{\eta}}) = \int_{X} \big(\Psi_f^{\univ}(\mathds{1}_{X_{\eta}}) \big)_{\mid \mathbb{A}^{d_1}_k \times X} \simeq \mathds{1}_{X_{\sigma}}(d_1)[2d_1] \otimes  \Psi_{f_{\mid X}}^{\univ}(\mathds{1}_{X_{\eta}}).
  \end{equation*}
\end{cor}
In addition to the results above, in the appendix, we prove a geometric version of the Beilinson equivalence (see \cite{beilinson-1987-2}), which is implicitly used throughout the paper. Let $\dcat_{\geo}^b(X) \subset \dcat_{\geo}^b(X)$ be the full sub-$\infty$-category of complexes of geometric origin, namely, the full thick subcategory of $\mathbf{D}^b_{\ct}(X)$ generated by complexes of the form $p_*\mathbb{Q}_Y$ with $p \colon Y \longrightarrow X$ a proper morphism.
\begin{theorem}
    Let $X$ be a $k$-variety, then the Beilinson realization functor
    \begin{equation*}
        \operatorname{real} \colon \dcat^b(\perv_{\geo}(X)) \longrightarrow \dcat^b_{\geo}(X)
    \end{equation*}
    is an equivalence of categories. Moreover, all the usual operations are well-restricted to $\dcat_{\geo}^b(-)$, e.g., six operations, (unipotent) nearby functors, Beilinson's gluing functors. Moreover, as for Nori motives, there are isomorphisms of functors
    \begin{align*}
     \Psi_f^{\anor} \simeq \Psi_f^{\an} \colon \dcat_{\geo}^b(X_{\eta}) \longrightarrow \dcat_{\geo}^b(X_{\sigma}) \\ 
       \log_f^{\an} \simeq  \Upsilon_f^{\anor} \simeq \Upsilon^{\an}_f \colon \dcat^b_{\geo}(X_{\eta}) \longrightarrow \dcat^b_{\geo}(X_{\sigma})
    \end{align*}
    for a given morphism $f \colon X \longrightarrow \mathbb{A}_k^1$.
\end{theorem}
The reader may realize that the second part of the theorem above can be viewed as a geometric avatar of theorem \ref{thm: first main theorem} and the first isomorphism $\Psi_f^{\anor} \simeq \Psi_f^{\an}$ is more or less a consequence of the compatibility of nearby functors under the Betti realizations in \cite{ayoub-2010} but this is not entirely true. The compatibility $\betti^* \Psi_f^{\anor} \longrightarrow \Psi_f^{\an} \betti^*$ is constructed via a direct morphism of algebraic derivators and does not make use of any theory of perverse sheaves and unipotent nearby cycles. Meanwhile, our isomorphism $\Psi_f^{\anor} \simeq \Psi_f^{\an}$ is obtained via isomorphisms on unipotent nearby functors $\Upsilon_f^{\anor} \simeq \log_f^{\an} \simeq \Upsilon_f^{\an}$ and the fact that on geometric motives, full nearby functors are "logarithmic" and all of these are done by using perverse sheaves and Beilinson's equivalence. 
\section{Ayoub motivic nearby functors}
\subsection{General facts on specialization systems} Let $X$ be a $k$-variety and $f \colon X \longrightarrow \mathbb{A}^1_k$ be a $k$-morphism. Let $\hbf \colon (\var_k)^{\op} \longrightarrow \operatorname{CAlg}(\operatorname{Cat}_{\infty}^{\textnormal{st}})$ be a coefficient system in the sense of \cite{drew+gallauer-2022} (in the language of triangulated categories, this is called \textit{stable homotopical $2$-functor} in \cite[Definition 1.4.1]{ayoub-thesis-1}). Let us consider the following diagram with cartesian squares
\begin{equation*}
    \begin{tikzcd}[sep=large]
       X _{\eta} \arrow[r,"j_f"]  \arrow[d,"f_{\eta}",swap] & X  \arrow[d,"f"] & X_{\sigma} \arrow[d,"f_{\sigma}"] \arrow[l,"i_f",swap] \\ 
        \eta \coloneqq \mathbb{G}_{m,k} \arrow[r,"j_{\id}"] & \mathbb{A}_k^1 & \sigma \coloneqq \Spec(k). \arrow[l,"i_{\id}",swap],
    \end{tikzcd}
    \end{equation*} 
where $i \colon \Spec(k) \longhookrightarrow \mathbb{A}_k^1$ is the zero section and $j \colon \gm_{m,k} \longhookrightarrow \mathbb{A}_k^1$ is its open complement. A \textit{specialization system} is a collection of $\infty$-functors 
\begin{equation*}
    \sp_f \colon \hbf(X_{\eta}) \longrightarrow \hbf(X_{\sigma})
\end{equation*}
together with (adjoint) transformations 
\begin{align*}
    \alpha_g \colon g_{\sigma}^*\sp_f &\longrightarrow \sp_{f \circ g} g_{\eta}^* \\
    \beta_g \colon \sp_f g_{\eta,*} & \longrightarrow g_{\sigma,*}\sp_{f \circ g}
\end{align*}
for another $k$-morphism $g \colon Y \longrightarrow X$, so that $\alpha_g$ is an isomorphism if $g$ is smooth and $\beta_g$ is an isomorphism if $g$ is proper. A morphism of specialization systems is a collection of $\infty$-functors $\sp_f \longrightarrow \sp'_f$ associated with morphisms $f \colon X \longrightarrow \mathbb{A}_k^1$ such that the diagram
\begin{equation*}
    \begin{tikzcd}[sep=large]
        g_{\sigma}^*\sp_f \arrow[r] \arrow[d] &  g_{\sigma}^*\sp'_f  \arrow[d]  \\ 
        \sp_{f \circ g} g_{\eta}^* \arrow[r] & \sp'_{f \circ g} g_{\eta}^* 
    \end{tikzcd}
\end{equation*}
is commutative. Inside $\hbf(X)$, we define the full, stable sub-$\infty$-category $\hbf_{\geo}(X)$ of geometric motives generated by motives of the form $p_*(\mathds{1}_Y)(n)$ with $p \colon Y \longrightarrow X$ a proper morphism. Clearly, geometric motives are preserved under realization functors. Let $R \colon \hbf^1 \longrightarrow \hbf^2$ be a morphism of six-functor formalisms. There exists an induced morphism 
\begin{equation*}
    R \colon \hbf^1_{\geo} \longrightarrow \hbf^2_{\geo}. 
\end{equation*}
We have the following observation: let $\sp^i$ be a specialization system on $\hbf^i$ (with $i \in \left \{1,2 \right \}$) and $R \colon \hbf^1 \longrightarrow \hbf^2$ be a morphism of coefficient systems so that there are natural commutative squares
\begin{equation*}
    \begin{tikzcd}[sep=large]
        \hbf^1(X_{\eta}) \arrow[r,"R_{X_{\eta}}"] \arrow[d,"\sp_f^1",swap] & \hbf^2(X_{\eta}) \arrow[d,"\sp_f^2"]\\ 
        \hbf^2(X_{\sigma}) \arrow[r,"R_{X_{\sigma}}"] & \hbf^2(X_{\sigma}) 
    \end{tikzcd}
\end{equation*}
then if $\sp^1$ preserves geometric motives then $\sp^2$ preserves geometric motives as well. Let us consider some examples that are of our main interest:
\begin{ex} \label{ex: examples of sp}
    \begin{enumerate}
        \item Let $\hbf(-) = \daetct(-)$, the by \cite[Th\'eorème 10.2 et Th\'eorème 10.6]{ayoub-2014} (the definitions are recalled below), both motivic (unipotent) nearby functors are geometric.
        \item Let $\hbf(-) = \dcatct^b(-)$, then we have the ussual nearby functor $\Psi_f^{\an} \colon \dcatct^b(X_{\eta}) \longrightarrow \dcatct^b(X_{\sigma})$ (for a short reminder, see section 3.2 below). Let us consider first $f = e_n \colon $ the $n$-power morphism. Under the monodromy correspondence, $\exp_{X_{n,\eta},*}\exp_{X_{n,\eta}}^*(\mathds{1})$ corresponds to the $R=k[T,T^{-1}]$-module $\operatorname{Hom}_R(k,\Hom_k(R,R/(T^n-1)))$ and hence $\Psi_{e_n}^{\an}$ corresponds to $R/(T^n-1)$. In particular, $\Psi_{e_n}^{\an}(\mathds{1})$ and $\Upsilon_{e_n}^{\an}(\mathds{1})$ are  geometric complexes. 
        \item More generally, let $X$ be a smooth $k$-variety and $f \colon X \longrightarrow \mathbb{A}^1_k$ be a morphism so that $f = ug^e$ with $u \in \mathcal{O}_X(X)^{\times}$ and $g$ a generator of the defining ideal of $(X_{\sigma})_{\textnormal{red}} \subset X$ then thanks to \cite[Th\'eorème 4.9]{ayoub-2010}, we have that $\betti_{X_{\sigma}}\Psi_f^{\anor}(\mathds{1}_{X_{\sigma}}) = \Psi_f^{\an}(\mathds{1}_{X_{\sigma}^{\an}})$ is geometric since $\Psi_f^{\anor}(\mathds{1}_{X_{\eta}})$ is geometric. 
    \end{enumerate}
\end{ex}

We take the following important theorem of Ayoub (see \cite{ayoub-thesis-2}\cite{ayoub+florian+julien-2017}) for granted. 
\begin{theorem}[Semi-stable reduction]
    Assume that $\hbf$ is $\mathbb{Q}$-linear and separated and $\sp$ is a specialization system on $\hbf$. Assume that $X$ is smooth and $X_{\sigma}$ is a divisor with normal crossings with irreducible components $D_1,...,D_n$. Let $I \subset \left \{1,...,n\right \}$ and set $D_I = \cap_{i \in I} D_i$, $D_I^{\circ} = D_I \setminus (\bigcup_{i \notin I} D_i)$. Let 
\begin{equation*}
    D_I^{\circ} \overset{v_I}{\longrightarrow} D_I \overset{u_I}{\longrightarrow} X_{\sigma}
\end{equation*}
be canonical inclusions, then the morphism 
    \begin{equation*}
    u_I^* \sp_f f_{\eta}^* \longrightarrow v_{I,*}v^*_I u^*_I\sp_f f_{\eta}^*
    \end{equation*}
    is an isomorphism. 
\end{theorem}

\subsection{Ayoub's motivic (unipotent) nearby functors}
In this subsection, we study (motivic) nearby functors, following the content of \cite{ayoub-thesis-2} with a view towards applications to nearby functors on Nori motives. In \cite{ayoub-thesis-2}, Ayoub defines and studies motivic (unipotent) nearby functors associated with regular functions $f \colon X \longrightarrow \mathbb{A}_k^1$. Let us recall the definition. Let $\hbf \colon (\var_k)^{\op} \longrightarrow \operatorname{CAlg}(\operatorname{Cat}_{\infty}^{\textnormal{st}})$ be a coefficient system in the sense of \cite{drew+gallauer-2022}. Let $f \colon X\longrightarrow \mathbb{A}_k^1$ be a morphism of $k$-varieties, let us consider the following diagram with cartesian squares
\begin{equation*}
    \begin{tikzcd}[sep=large]
       X _{\eta} \arrow[r,"j_f"]  \arrow[d,"f_{\eta}",swap] & X  \arrow[d,"f"] & X_{\sigma} \arrow[d,"f_{\sigma}"] \arrow[l,"i_f",swap] \\ 
        \eta \coloneqq \mathbb{G}_{m,k} \arrow[r,"j_{\id}"] & \mathbb{A}_k^1 & \sigma \coloneqq \Spec(k). \arrow[l,"i_{\id}",swap],
    \end{tikzcd}
    \end{equation*} 
where $i \colon \Spec(k) \longhookrightarrow \mathbb{A}_k^1$ is the zero section and $j \colon \gm_{m,k} \longhookrightarrow \mathbb{A}_k^1$ is its open complement. Let $\Delta$ be the category of finite ordinals $\mathbf{n} = \left \{0 < 1 < \cdots < n \right \}$ and  $\mathbb{N}^{\times}$ be the set of non-zero natural numbers viewed as a category whose objects are non-zero natural numbers and morphisms are defined via the opposite of the division relation.  In \cite[Definition 3.5.1]{ayoub-thesis-1}, Ayoub defines a diagram of $\gm_{m,k}$-schemes $\rscr \colon \Delta \times \mathbb{N}^{\times} \longrightarrow \var_{\gm_{m,k}}$ such that $\mathscr{R}(\mathbf{n},r) = \gm_{m,k} \times_k(\gm_{m,k})^n$. The structural morphism is given by the composition
\begin{equation*}
    \gm_{m,k} \times_k (\gm_{m,k})^n \overset{\pr_1}{\longrightarrow} \gm_{m,k} \overset{e_r}{\longrightarrow} \gm_{m,k}, 
\end{equation*}
in which $\pr_1$ is the projection on the first factor and the second one is the $r$-power morphism and morphisms between $\rscr(\mathbf{n}_1,r_1) \longrightarrow \rscr(\mathbf{n}_2,r_2)$ are given simplicially in the first factor and by divisibility in the second factor. Let us assume that each $\hbf(X)$ is presentable and $f^*$ preserves compact objects (in particular, this implies $f_*$ preserves (homotopical) colimits by \cite{ayoub-thesis-1}). The \textit{motivic nearby functor} and the \textit{motivic unipotent nearby functor} are given by
\begin{equation*}
    \begin{split} \Psi_f^{\anor} \colon \hbf(X_{\sigma}) & \longrightarrow \hbf(X_{\sigma})\\ 
    M & \longmapsto i^*j_*(M \otimes f_{\eta}^*(\uscr^{\mot})),
    \end{split} \ \ \ \ \ \ \begin{split} 
   \Upsilon_f^{\anor} \colon \hbf(X_{\sigma}) & \longrightarrow \hbf(X_{\sigma})\\ 
    M & \longmapsto i^*j_*(M \otimes f_{\eta}^*(\uscr^{\textnormal{un}})),
    \end{split}
\end{equation*}
respectively, where $\uscr^{\mot} = \colim_{\Delta \times \mathbb{N}^{\times}} \mnor(\rscr(\mathbf{n},r))$ \footnote{For $\hbf = \daet$, under the Betti realization and the pullback of the unit section, one has that $\Spec(1^*\betti_k^*(\uscr)) \simeq \hat{\mathbb{Z}}(1) \times \mathbb{G}_a$ thanks to \cite[Corollaire 2.19]{ayoub-hopf2}.} is the homotopy colimit of motives of schemes in the diagram $\rscr$ and $\uscr^{\textnormal{un}} = \colim_{\Delta} \mnor(\rscr(\mathbf{n},1))$. Alternatively, one can follow \cite[Definition 4.11]{ayoub-nearby-cycles} and write the full nearby functor as follows: for each $n \in \mathbb{N}$, let us form the cartesian square
\begin{equation*}
    \begin{tikzcd}[sep=large]
        X_n \arrow[d,"f_n"] \arrow[r,"e_n"] & X \arrow[d,"f"] \\ 
        \mathbb{A}_k^1 \arrow[r,"e_n"] & \mathbb{A}_k^1
    \end{tikzcd}
\end{equation*}
where $e_n \colon \mathbb{A}_k^1 \longrightarrow \mathbb{A}_k^1$ is the $n$-th power morphism. Then the motivic nearby functor can be given as
\begin{equation*}
    \Psi_f^{\anor} = \colim_{n \in \mathbb{N}} \ \Upsilon_{f_n}^{\anor}(e_n)_{\eta}^*, 
\end{equation*}
where the transitions are given as in \cite[Lemma 4.10]{ayoub-nearby-cycles}. Clearly, if $R \colon \hbf_1 \longrightarrow \hbf_2$ is a morphism of coefficient systems, there are natural isomorphisms
\begin{equation*}
    R_{X_{\sigma}} \circ \Psi_f^{\anor} \overset{\sim}{\longrightarrow} \Psi_f^{\anor} \circ R_{X_{\eta}} \ \ \ \ \ \ R_{X_{\sigma}} \circ \Upsilon_f^{\anor} \overset{\sim}{\longrightarrow} \Upsilon_f^{\anor} \circ R_{X_{\eta}} 
\end{equation*}
Let us recollect the general properties of motivic (unipotent) nearby functors. In what follows, let us assume that $\hbf \in \left \{\daet(-),\dcat_{\geo}(-),\dn(-) \right \}$.
 \begin{prop} \label{prop: computation of nearby cycles}
    Let $C = \mathbb{A}_k^1 = \Spec(k[\pi])$ and $f \colon X \longrightarrow C$ be a morphism of $k$-varieties. Let us consider the following $C$-schemes 
\begin{equation*}
    e_n \colon C[t]/(t^n - \pi) \longrightarrow C \ \ \ \ \textnormal{and} \ \ \ \  
    e_n' \colon C[t,u,u^{-1}]/(t^n - u \pi) \longrightarrow C \\ 
\end{equation*}
The following statements hold true:
\begin{enumerate}
    \item There are canonical isomorphisms 
\begin{equation*}
    \Psi_{e_n}^{\anor}(\mathds{1}) \simeq (t_n)_{\sigma,*}(\mathds{1}) \ \ \ \ \textnormal{and} \ \ \ \  \Psi_{e_n'}^{\anor}(\mathds{1}) \simeq  (t_n')_{\sigma,*}(\mathds{1})
\end{equation*}
where
\begin{align*}
    t_n \colon \Spec(k[t]/(t^n - 1)) & \longrightarrow \Spec(k) \\ 
    t_n' \colon \Spec(k[t,u,u^{-1}]/(t^n-u)) & \longrightarrow \Spec(k[u,u^{-1}]).
\end{align*}
\item Let $X$ be a smooth $k$-variety and $f \colon X \longrightarrow \mathbb{A}_k^1$ be a $k$-morphism. Assume that $D = (X_{\sigma})_{\textnormal{red}}$ is a smooth $k$-scheme and is a principal divisor with $g \in \mathcal{O}_X$ a fixed generator of the defining ideal. Assume further that there exists $u \in \mathcal{O}_X(X)^{\times}$ and $n \in \mathbb{N}$ such that $f = ug^n$. Let us consider the \'etale cover
\begin{equation*}
    r_n \colon D_n = D[t]/(t^n - u_{\mid D}) \longrightarrow D, 
\end{equation*}
then one has a canonical isomorphism
\begin{equation*}
    \Psi_f^{\anor}(\mathds{1}_{X_{\eta}}) = (r_n)_*(\mathds{1}_{D_n}) \ \ \ \ \textnormal{and} \ \ \ \ \Upsilon_f^{\anor}(\mathds{1}_{X_{\eta}}) = \mathds{1}_{X_{\sigma}}.
\end{equation*}
\end{enumerate}
\end{prop}

\begin{proof}
    Regard the case of \'etale motives, the first property is given in \cite[Proposition 3.5.12]{ayoub-thesis-2}, the second property is given in \cite[Th\'eorème 10.6]{ayoub-2014} (see also \cite[Proposition 3.4]{ayoub+florian+julien-2017}). We only need the properties that $\Psi_{\id}(\mathds{1}_{\eta}) = \mathds{1}_{\sigma}$ and $\Psi_f \simeq \Psi_{f_n}(e_n)_{\eta}^*$. For Nori motives and geometric complexes, we use the fact that the Nori realization (constructed in \cite{tubach-2025}) and the Betti realization (constructed in \cite{ayoub-2010}) are unital and commute with six operations. 
\end{proof}

\begin{prop} \label{prop: functoriality of nearby cycles}
    Let $g \colon Y \longrightarrow X, f \colon X \longrightarrow \mathbb{A}_k^1$ be morphisms of $k$-varieties, the following statements hold true for $\Psi = \Psi^{\anor}$:
    \begin{enumerate}
    \item There exists a natural transformation
    \begin{equation*}
        \alpha_g \colon g_{\sigma}^*\Psi_f \longrightarrow \Psi_{f \circ g} g_{\eta}^* \ \ \ \ g_{\sigma}^*\Upsilon_f \longrightarrow \Upsilon_{f \circ g} g_{\eta}^*
    \end{equation*}
    which is an isomorphism if $g$ is smooth.
    \item The adjoint of $\alpha_g$, denoted
    \begin{equation*}
        \beta_g \colon \Psi_fg_{\eta,*} \longrightarrow g_{\sigma,*}\Psi_{f \circ g} \ \ \ \  \Upsilon_fg_{\eta,*} \longrightarrow g_{\sigma,*}\Upsilon_{f \circ g}
    \end{equation*}
    is an isomorphism if $g$ is proper.
    \item  There exists a natural isomorphism $\mathbb{D}_{\xfrak_{\sigma}}  \circ \Psi_f \longrightarrow \Psi_f \circ \mathbb{D}_{\xfrak_{\eta}}$.
    \item Let $f \colon X \longrightarrow \mathbb{A}_k^1, g \colon Y \longrightarrow \mathbb{A}_k^1$ be morphisms of $k$-varieties, there exists a natural isomorphism
    \begin{equation*}
        \Psi_f(-) \boxtimes \Psi_g(-) \overset{\simeq}{\longrightarrow}\Psi_{f \times g}(- \boxtimes -).
    \end{equation*}
    \item There is a natural transformation 
    \begin{equation*}
        \Psi_f(-) \otimes \Psi_f(-) \longrightarrow \Psi_f(- \otimes -)
    \end{equation*}
    and moreover, the functor $\Psi_{\id}$ is unital monoidal.
    \end{enumerate}
\end{prop}
\begin{proof}
    One can find in \cite[Chapitre 3]{ayoub-thesis-2} all the necessary materials showing these properties.
\end{proof}

\begin{prop} \label{prop: compactness of Ayoub's functors}
Let $X$ be a $k$-variety, the category $\hbf(X)$ is compactly generated by motives of the form $p_*(\mathds{1}_Z)(n)$ with $p \colon Z \longrightarrow X$ a proper morphism and $n \in \mathbb{Z}$. In particular, the functors $\Psi_f^{\anor},\Upsilon_f^{\anor}$ preserves compact motives
\end{prop}

\begin{proof}
    For $\daet(X)$, this is \cite[Proposition 3.19]{ayoub-2014} and \cite[Proposition 2.2.27]{ayoub-thesis-2}. For $\dcat_{\geo}(X)$, this is obvious. By \cite{tubach-2025}, $\dn(X)$ is compactly generated by motives of the form $f_{\#}(\mathds{1}_Y)(i)$ with $Y$ smooth over $X$ and $i \in \mathbb{Z}$. Using descent, one can reduce the generators to those in statements. Alternatively, using what known for $\daet$, we can use \cite[Lemma 2.6]{iwanari}. The functors $\Psi_f^{\anor},\Upsilon_f^{\anor}$ preserve compact motives on $\daet(-)$. Using realization and the fact that $\dn(X),\dcat_{\geo}$ are compactly generated by motives of the form $p_*(\mathds{1}_Z(n))$, it suffices to check that $\Psi_f^{\anor}(p_*\mathds{1}_Z(n)),\Upsilon_f^{\anor}(p_*\mathds{1}_Z(n))$ are compact and this is obvious since they lie in the essential realization functors and the fact that realization functors preserve compact motives. 
\end{proof}

The following proposition is essentially due to Ayoub in \cite{ayoub-thesis-2}\cite{ayoub-nearby-cycles} (see also \cite{florian+julien-2013}\cite{ayoub+florian+julien-2017}\cite{florian+julien-2021} and \cite{bang-thesis}). It holds true for any $\mathbb{Q}$-linear, separated coefficient system and in fact for our later applications, it suffices to show it for \'etale motives and claim the result using realization functors from \'etale motives to other theories. We present here the complete for the convenience of the reader. 
\begin{prop}  \label{prop: logarithm of nearby cycles}
Let $M \in \hbf_{\ct}(X_{\eta})$ be a compact motive, then there exists an integer $n_0=n_0(M)$ such that $\Upsilon_{f_n}^{\anor}(e_n)_{\eta}^*(M) \simeq \Psi_f^{\anor}(M)$ for any $n$ divisible by $n_0$. 
\end{prop}

\begin{proof}
All geometric motives are built by taking cofibers, retracts, finite sums, shifts of motives of the form $M = g_{\eta}(\mathds{1}_{Y_{\eta}})$ for a proper morphism $g \colon Y \longrightarrow X$ so it suffices to assume that $M$ has this form. It is obvious that there is a commutative diagram
\begin{equation*}
\begin{tikzcd}[sep=large]
\Upsilon_{f_n}^{\anor}(e_n)_{\eta}^*\big((g_{\eta})_*(\mathds{1}_{Y_{\eta}}) \big) \arrow[r,"\sim"] \arrow[d] &   (g_{\sigma})_*\Upsilon_{f_n \circ g_n}^{\anor}(e_n)_{\eta}^*\big(\mathds{1}_{Y_{\eta}} \big) \arrow[d] \\
    \Psi_f\big((g_{\eta})_*(\mathds{1}_{Y_{\eta}}) \big) \arrow[r,"\sim"] &   (g_{\sigma})_*\Psi_{f \circ g}\big(\mathds{1}_{Y_{\eta}} \big) 
    \end{tikzcd}
\end{equation*}
Thus it is enough to prove the statement for $M = \mathds{1}_{X_{\eta}}$. At this point, let us consider a resolution of singularities 
\begin{equation*}
    h \colon (Y,E) \longrightarrow (X,X_{\sigma})
\end{equation*}
so that $h_{\mid Y \setminus E} \colon Y \setminus E \longrightarrow X \setminus X_{\sigma} = X_{\eta}$ is an isomorphism and $E$ is a divisor with normal crossings. Since $h_{\eta,*}(\mathds{1}_{Y_{\eta}}) \simeq \mathds{1}_{X_{\eta}}$ and by the proper base change theorem
\begin{equation*}
    \sp_f(\mathds{1}_{X_{\eta}}) \simeq \sp_f\big( h_{\eta,*}(\mathds{1}_{Y_{\eta}})\big) \simeq h_{\sigma,*}\sp_{f \circ h}(\mathds{1}_{Y_{\eta}}).
\end{equation*}
Therefore, we can assume that $X_{\sigma}$ is a divisor with normal crossings. Let $(D_i)_{i \in I}$ be irreducible components of $X_{\sigma}$. We use the notation of proposition \ref{prop: functoriality of nearby cycles}. The Mayer-Vietoris sequence shows that $\Psi^{\anor}_f(\mathds{1}_{X_{\eta}}),\Upsilon^{\anor}_{f_n}(\mathds{1}_{(X_n)_{\eta}})$ belong to the subcategory generated by $u_{J,*} u_J^*\Psi^{\anor}_f(\mathds{1}_{X_{\eta}}),u_{J,*} u_J^*\Upsilon^{\anor}_{f_n}(\mathds{1}_{(X_n)_{\eta}})$, respectively, with $J \subset I$ nonempty. This reduces to proving that 
\begin{equation*}
    u_{J,*} u_J^*\Upsilon^{\anor}_{f_n}(\mathds{1}_{(X_n)_{\eta}}) \longrightarrow  u_{J,*} u_J^*\Psi^{\anor}_f(\mathds{1}_{X_{\eta}})
\end{equation*} 
is an isomorphism. Let us demonstrate that it is safe to suppose that $\left |J \right| = 1$. We denote by $\sp_f$ to indicate either $\Upsilon_{f_n}^{\anor}$ or $\Psi_f^{\anor}$. We consider the diagram 
\begin{equation*}
    \begin{tikzcd}[sep=huge]
        E(J) \arrow[d] \arrow[r,"e",hook] & X(J)_{\sigma} \arrow[r]  \arrow[d,"\epsilon(J)_{\sigma}"] & X(J) \arrow[d,"\epsilon(J)"] \\ 
        D_J \arrow[r,"u_J"] & X_{\sigma} \arrow[r,"i_f"] & X
    \end{tikzcd}
\end{equation*}
where $\epsilon(J) \colon X(J) \longrightarrow X$ is the blow-up in $X$ with center $D_J$ and exceptional divisor $E(J)$. The special fiber $X(J)_{\sigma}$ is a divisor with normal crossings whose irreducible components are $E(J)$ and strict transforms of $D_i$ for $i \in I$. We have 
\begin{align*}
    u_{J,*}u_J^*\sp_f(\mathds{1}_{X_{\eta}}) & \simeq u_{J,!}u_J^*(\epsilon(J)_{\sigma})_*\sp_{f \circ \epsilon(J)}(\mathds{1}_{X(J)_{\eta}}) \\ 
    & \simeq (\epsilon(J)_{\sigma})_*e_*e^*\sp_{f \circ \epsilon(J)}(\mathds{1}_{X(J)_{\eta}}).
\end{align*}
Again, we reduce the problem to showing that $e_*e^*\Upsilon^{\anor}_{f_n \circ \epsilon(J)_n}(\mathds{1}_{X(J)_{\eta}}) \longrightarrow e_*e^*\Psi^{\anor}_{f \circ \epsilon(J)}(\mathds{1}_{X(J)_{\eta}})$ is an isomorphism and hence we can assume that $J = \left \{ i \right \} \subset I$. By the semi-stable reduction property, there is a canonical isomorphism 
\begin{equation*}
  u_{i,*} v_{i,*}v_i^*u_i^*\sp_f(\mathds{1}_{X_{\eta}}) \simeq u_{i,*}u_i^*\sp_f(\mathds{1}_{X_{\eta}})
\end{equation*}
with $v_i \colon D_i^{\circ} \longhookrightarrow D_i$ the canonical open immersion. Consequently, it boils down to proving that $v_i^!u_i^!\sp_f(\mathds{1}_{X_{\eta}})$ is compact. Let us denote by $\tau_i \colon X_i = X \setminus \left( \bigcup_{j \neq i} D_j \right)_{\textnormal{red}} \longhookrightarrow X$ the canonical open immersion. With this, one has that $((X_i)_{\sigma})_{\textnormal{red}} \simeq D_i^{\circ}$ and there is a canonical factorization 
\begin{equation*}
    \begin{tikzcd}[sep=large]
        ((X_i)_{\sigma})_{\textnormal{red}} \simeq D_i^{\circ} \arrow[r] \arrow[rr,bend left = 25,"u_i \circ v_i"] \arrow[r,"r_i"] & (X_i)_{\sigma} \arrow[r,"\tau_i"] & X_{\sigma}.
    \end{tikzcd}
\end{equation*}
We then have that
\begin{equation*}
      v_i^*u_i^*\sp_f(\mathds{1}_{X_{\eta}}) \simeq r_i^*\sp_{f_{\mid X_i}}(\mathds{1}_{(X_i)_{\eta}})
\end{equation*}
(since $\tau_i$ is an open immersion). On the open subscheme $X_i \longhookrightarrow X$, all the regular functions $t_j$ with $j \neq i$ are invertible. Therefore we can simplify the expression of $f$ in the ring $\Gamma(X_i,\mathcal{O}_{X_i})$ as $f = vt_i^{N_i}$ with $v \in \Gamma(X_i,\mathcal{O}^{\times}_{X_i})$. Now by $(2)$ in proposition \ref{prop: computation of nearby cycles}, the morphism $\Upsilon^{\anor}_{f_n}(\mathds{1}_{X_{\eta}}) \longrightarrow \Psi^{\anor}_f(\mathds{1}_{X_{\eta}})$ is an isomorphism and we finish the proof.
\end{proof}

The proposition \ref{prop: logarithm of nearby cycles} applies, for example, to geometric origin, compactness and any further property for which the standard one-branch calculations are known. In particular, the proof displays how an arbitrary constructible motive is assembled from the local normal-crossings models. Its proof yield the followng consequences, which are already known in \cite{ayoub-thesis-2}.
\begin{cor} \label{cor: criterion for compactness}
   Let $\sp$ be a specialization system, if $\sp_f(\mathds{1}_{X_{\eta}})$ is compact whenever $X$ satisfies (2) in proposition \ref{prop: computation of nearby cycles}, then $\sp_f(M)$ is compact for any $f$ and $M$ compact. 
\end{cor}

\begin{cor} \label{cor: criterion for iso of sp}
    Let $\sp \longrightarrow \sp'$ be a morphism of specialization systems. If $\sp_f(\mathds{1}_{X_{\eta}}) \longrightarrow \sp_f'(\mathds{1}_{X_{\eta}})$ is an isomorphism whenever $X$ satisfies (2) in proposition \ref{prop: computation of nearby cycles}, then $\sp \longrightarrow \sp'$ is an isomorphism on constructible motives. If moreover, $\sp,\sp'$ commute with direct sums, then $\sp \longrightarrow \sp'$ is an isomorphism on all motives. 
\end{cor}

\section{The logarithmic specialization system}
This section serves as prerequisites for the next one where we would like to define the "universal" monodromy operator on perverse Nori motives. By the universal property, this is tantamout to proving two monodromy sequences on \'etale motives and on constructible complexes of geometric origin compatible. In order to do this, we construct the logarithmic specialization system as an intermediate specialization system. These specialization systems are introduced in \cite{ayoub-nearby-cycles}\cite{ayoub-thesis-2}\cite{ayoub-2014} and will be useful here since on perverse sheaves, they are parallel to Beilinson's construction of unipontent nearby cycles. 
\subsection{The Kummer motives} 
In \cite{ayoub-thesis-2}\cite{ayoub-2010}\cite{ayoub-nearby-cycles}, Ayoub defines the motivic monodromy sequence
\begin{equation*}
    i^*j_* \longrightarrow \Upsilon_f \overset{N}{\longrightarrow} \Upsilon_f(-1) \longrightarrow +1
\end{equation*}
where $N \colon \Upsilon_f \longrightarrow \Upsilon_f(-1)$ is the so-called \textit{monodromy operator}. In the classical setting, this is the logarithm of the unipotent part of the monodromy action. The goal of this section is to prove that one can lift this sequence to the category of Nori motives. Let us recall the construction of the monodromy operator. Let $\hbf \in \left \{\daet(),\dn(-),\dcat_{\geo}(-) \right \}$. By the \cite[Proposition 11.1]{ayoub-2014}, there is a unique extension in $\daet(\gm_m)$ of the form
\begin{equation*}
    \mathds{1}_{\gm_m} \longrightarrow \mathcal{K} \longrightarrow \mathds{1}_{\gm_m}(-1) \longrightarrow +1
\end{equation*}
called the \textit{Kummer extension} (the Kummer extension is a Tate motive so in particular, over number fields, one can use the structure of a motivic $t$-structure on Tate motives due to Levine in \cite{levine-1993} to show the existence of $\mathcal{K}$). In particular, under the Nori and the Betti realization, there are corresponding cofiber sequences
\begin{align*}
     \mathds{1}_{\gm_m} \longrightarrow \mathcal{K}^{\nori}\longrightarrow \mathds{1}_{\gm_m}(-1) \longrightarrow +1 \\ 
      \mathds{1}_{\gm_m} \longrightarrow \mathcal{K}^{\an} \longrightarrow \mathds{1}_{\gm_m}(-1) \longrightarrow +1
\end{align*}
where $\kcal^{\nori}=\nori^*_{\gm_{m,k}}(\kcal)$ and $\kcal^{\nori} = \betti^*_{\gm_{m,k}}(\kcal)$ are images of the Kummer motive. The complex $\kcal^{\an}$ is necessarily unique thanks to the following lemma.  
\begin{lem} \label{lem: monodromy of Kummer motive}
    In $\dcat_{\geo}(\gm_{m,k})$, there is a unique nontrivial extension (of local systems) the form 
    \begin{equation*}
    \mathds{1}_{\gm_m} \longrightarrow \mathcal{K}^{\an} \longrightarrow \mathds{1}_{\gm_m}(-1) \longrightarrow +1
\end{equation*}
    and moreover, after choosing a canonical basis, the complex $\mathcal{K}^{\an}$ (indeed, a sheaf) has monodromy action given by 
        \begin{equation*}
            L_2 =  \begin{pmatrix}
1 & -1 \\
0  & 1 \\
\end{pmatrix}.
        \end{equation*} and there is a canonical isomorphism $\betti^*(\kcal) = \kcal^{\an}$,
\end{lem}
\begin{proof}
    The last part follows from the first part since $\mathcal{K},\mathcal{K}^{\an}$ are uniquely determined and the Betti realization is unital and the second part is standard linear algebra, given the first part. To show the existence and the uniqueness of $\mathcal{K}^{\an}$, we can show its existence and uniqueness in the heart of the constructible $t$-structure and this corresponds to an element of $\operatorname{Ext}^1(\mathbb{Q}(1),\mathbb{Q}) \simeq H^1(\gm_{m,k},\mathbb{Q}) \simeq \mathbb{Q}$, which is one-dimensional. The sheaves $\mathbb{Q}(1),\mathbb{Q}$ have trivial monodromy so the monodromy on $\mathcal{K}^{\an}$ is of the form $T =  \begin{pmatrix}
1 & a \\
0  & 1 \\
\end{pmatrix}$. Since $\mathcal{K}^{\an}$ is nontrivial it implies that $a \neq 0$ and hence after rescaling the basis, $T$ has the canonical Jordan form. 
\end{proof}

\begin{defn}
We fix $\hbf$ and its corresponding Kummer motive $\kcal$. Let $f \colon X \longrightarrow \mathbb{A}_k^1$ be a morphism of $k$-varieties, the \textit{logarithmic specialization system} 
    \begin{equation*}
    \operatorname{log}_f(M) \simeq \colim_{n \in \mathbb{N}} \ i^*j_*(M \otimes f_{\eta}^*\mathrm{Log}^{\vee}_n) \simeq i^*j_*(M \otimes f_{\eta}^*( \mathrm{Log}^{\vee})) \colon \hbf(X_{\eta}) \longrightarrow \hbf(X_{\sigma})
\end{equation*}
in which $\mathrm{Log}^{\vee}_n = \operatorname{Sym}^n(\mathcal{K})$ and $\mathrm{Log}^{\vee} = \operatorname{Hocolim} \ \mathrm{Log}^{\vee}_n$. The corresponding functors for $\daet(-),\dn(-),\dcat_{\geo}$ are denoted by $\log_f,\log_f^{\nori},\log_f^{\an}(-)$, respectively. It is clear that these functors are compatible under realizations.
\end{defn}

\begin{cor} \label{cor: log sp preserves compact motives}
     The functor $\log_f^{\an} \colon \hbf(X_{\eta}) \longrightarrow \hbf(X_{\sigma})$ preserves compact motives. In particular, there is an induced commutative square
    \begin{equation*}
        \begin{tikzcd}[sep=large]
            \daetct(X_{\eta}) \arrow[d,"\log_f \simeq \Upsilon_f",swap] \arrow[r,"\betti^*_{X_{\eta}}"] & \dcat^b_{\geo}(X_{\eta}) \arrow[d,"\log_f^{\an}"] \\ 
            \daetct(X_{\sigma}) \arrow[r,"\betti^*_{X_{\sigma}}"]  & \dcat^b_{\geo}(X_{\eta}).
        \end{tikzcd}
    \end{equation*}
    In other words, the logarithmic specialization systems are compatible under the Betti realizations. 
\end{cor}

\begin{proof}
    By \ref{cor: criterion for compactness}, to check the compactness, we just have to know the value of $\log_f^{\an}(\mathds{1}_{X_{\eta}}^{\an})$ with $X$ be a smooth $k$-variety and $f = ug^e$ with $u \in \mathcal{O}_X(X)^{\times}$ and $g$ a generator of the defining ideal of $(X_{\sigma})_{\textnormal{red}} \subset X$. Given \cite[Th\'eorème 11.14]{ayoub-2014} and example \ref{ex: examples of sp}, we win. 
\end{proof}
By \cite[(120)]{ayoub-2014}, there exist distinguished triangles 
 \begin{equation*}
     \mathrm{Log}_{m-1}^{\vee} \longrightarrow \mathrm{Log}^{\vee} \overset{N_m}{\longrightarrow} \mathrm{Log}^{\vee}(-m) \longrightarrow +1 
 \end{equation*}
By taking homotopy colimits, we obtain a cofiber sequence
\begin{equation*}
    i^*j_*(M \otimes f_{\eta}^*\mathrm{Log}_{m-1}^{\vee}) \longrightarrow \log_f \overset{N_m}{\longrightarrow} \log_f(-m) \longrightarrow +1.
\end{equation*}
For $m=1$, we obtain the monodromy sequence
\begin{equation*}
    i^*j_* \longrightarrow \log_f \overset{N = N_1}{\longrightarrow} \log_f(-1) \longrightarrow +1.
\end{equation*}
We have the following lemma.
\begin{lem} \label{lem: nilpotency of monodromy operator}
    There is an isomorphism of functors $m! \cdot N_m = N^{\circ m}$. The operator $N \colon \log_f \longrightarrow \log_f(-1)$ is nilpotent on compact motives and consequently, the operator $N_m \colon \log_f \longrightarrow \log_f(-m)$ is nilpotent on compact motives. Moreover, if $M$ is a compact motive, then the sequence
    \begin{equation*}
    i^*j_*(M \otimes f_{\eta}^*\mathrm{Log}_{m-1}^{\vee}) \longrightarrow \log_f(M) \overset{N_m}{\longrightarrow} \log_f(M)(-m) \longrightarrow +1.
    \end{equation*}
    is canonically split for sufficiently large $m=m(M) \geq 0$. 
\end{lem}
\begin{proof}
    For \'etale motives, these are \cite[Lemme 11.10]{ayoub-2014}, \cite[Théorème 11.16]{ayoub-2014} (see also \cite[Corollaire 3.6.49]{ayoub-thesis-2}). For $\dcat_{\geo}(-)$ and $\dn(-)$, we use the realization functors.
\end{proof}
 \begin{rmk} If $\hbf$ is $\daet(-)$ or $\dcat_{\geo}(-)$, then it is necessary that $N_m$ is nilpotent on constructible motives and uniquely determined. Indeed, the uniqueness of $N_m$ is guaranteed once we can prove \cite[Lemme 11.9]{ayoub-2014} for complexes of geometric origin, which is again based on \cite[Proposition 11.1]{ayoub-2014}. This boils down to proving that $\Hom_{\dcat_{\geo}(\gm_{m,k})}(\mathds{1},\mathds{1}(m)[-1])=0$, which is trivial. 
 \end{rmk}
 
\subsection{The logarithmic specialization systems on complexes of geometric origin}
 In \cite{ayoub-2010}, Ayoub proves that motivic nearby functors of \'etale motives and analytic sheaves are compatible under the Betti realization. Without a doubt that the same result holds true for motivic unipotent nearby functors and for unipotent nearby functors, one should expect a further compatibility of monodromy triangles. Let us review the construction of the unipotent nearby functor in the classical setting and the logarithm of the monodromy action. Let $f \colon X \longrightarrow \mathbb{A}_k^1$ be a morphism of $k$-varieties. We consider the following diagram
 \begin{equation*}
    \begin{tikzcd}[sep=large]
      \widetilde{X}_{\eta} \arrow[r,"\operatorname{exp}"] & X_{\eta} \arrow[r,"j_f"]  \arrow[d,"f_{\eta}",swap] & X  \arrow[d,"f"] & X_{\sigma} \arrow[d,"f_{\sigma}"] \arrow[l,"i_f",swap] \\ 
        \widetilde{\gm_{m,k}^{\an}} \arrow[r,"\operatorname{exp}"] &  \mathbb{G}_{m,k}^{\an} \arrow[r,"j_{\id}"] & \mathbb{A}_k^{1,\an} &  \Spec(k)^{\an}. \arrow[l,"i_{\id}",swap],
    \end{tikzcd}
    \end{equation*} 
where $f^{\an} \colon X^{\an} \longrightarrow \mathbb{A}_k^{1,\an} = \mathbb{C}$ is the associated function of analytic varieties, $\widetilde{\gm_{m,k}^{\an}}$ is the universal cover of $\gm_{m,k}^{\an}$. The full nearby functor is defined as 
 \begin{align*}
     \Psi_f^{\an} \colon \dcatct^b(X_{\eta}) & \longrightarrow \dcatct^b(X_{\sigma}) \\ 
     M & \longmapsto i^*j_*\operatorname{exp}_*\operatorname{exp}^*(M)
 \end{align*}
 (we note that this functor is defined on $\dcat^+(X_{\eta}) \longrightarrow \dcat^+(X_{\sigma})$, its constructibility is a nontrivial result). Any $g \in \pi_1(\gm_{m,k}) \simeq \mathbb{Z}$ induces an automorphism $g \colon  \widetilde{X}_{\eta} \overset{\sim}{\longrightarrow}  \widetilde{X}_{\eta}$ and hence acts naturally on the full nearby functor $\Psi_f^{\an}$ by $i^*j_*\operatorname{exp}_*\operatorname{exp}^*(g)$. In particular, the canonical generator $g_1 \colon t \mapsto \exp(2\pi i t)$ induces a \textit{monodromy action} or \textit{monodromy automorphism}
 \begin{equation*}
     T \colon \Psi_f^{\an}(M) \longrightarrow \Psi_f^{\an}(M)
 \end{equation*}
 for any $M \in \dcatct^b(X_{\eta})$ and hence for any $M \in \dcat^b_{\geo}(X_{\eta})$. The category $\dcat^b_{\geo}(X_{\eta})$ is a \textit{Krull-Schmidt category} so it induces a decomposition
 \begin{equation*}
     \Psi_f^{\an}(M) = \bigoplus_{\alpha \in \mathbb{Q}[x] \ \textnormal{irreducible,monic}} \Psi_f^{\an,\alpha}(M).
 \end{equation*}
 The functor $\Psi_f^{\an,x-1}(M) = \Upsilon_f^{\an}(M)$ ($\alpha(x)=x-1$) is called the \textit{unipotent nearby functor}. By the Beilinson's equivalence (see \cite{beilinson-1987-2} and the appendix, theorem \ref{thm: Beilinson's theorem for geometric complexes}), it sufffices to know $\Upsilon_f^{\an}[-1]$ on $\perv_{\geo}(X_{\eta})$ and then takes the derived functor. On perverse sheaves, we have the following construction due to Beilinson \cite{beilinson-1987-1} (see also \cite{reich-2010}\cite{morel-2018}\cite{achar-book}). Let $\mathcal{J}_n \in \operatorname{Loc}(\gm_{m,k})$ be the local system corresponding to the module $J_n=\mathbb{Q}[T,T^{-1}]/((T-1)^n)$ under the monodromy correspondence. By \cite[Proposition 4.4.3]{achar-book}, for $M \in \perv_{\geo}(X_{\eta})$, then for any sufficiently large $m \geq 0$, one has that
 \begin{equation*}
     \Upsilon_f^{\an}(M) = \phnor^{-1}(\chi_f(M \otimes f_{\eta}^*(\jcal_m(1-m)))).
 \end{equation*}
and the logarithm of the monodromy action $\Upsilon_f^{\an}(M) \longrightarrow \Upsilon_f^{\an}(M)(-1)$ is given by a morphism $\jcal_m(1) \longrightarrow \jcal_m$, which corresponds to the matrix 
\begin{equation*}
\begin{pmatrix}
0 & 0 & \cdots & 0 & 0 \\
1 & 0 & \cdots & 0 & 0 \\
0 & 1 & \ddots & \vdots & \vdots \\
\vdots & \ddots & \ddots & 0 & 0 \\
0 & \cdots & 0 & 1 & 0
\end{pmatrix}
\end{equation*}
under the monodromy correspondence. There is a canonical isomorphism of complexes $\mathrm{Log}_{n-1}^{\vee} \simeq \mathcal{J}_n(1-n)$. Indeed, by \cite[(4.4.2)]{achar-book}, there is a canonical short exact sequence of local systems
    \begin{equation*}
        0 \longrightarrow \mathbb{Q}_{\gm_{m,k}}(n-1) \longrightarrow \mathcal{J}_n \longrightarrow \mathcal{J}_{n-1} \longrightarrow 0.
    \end{equation*}
    In particular, $\mathcal{J}_0 = 0, \mathcal{J}_1 = \mathbb{Q}_{\gm_{m,k}}$ and $\jcal_2= \mathcal{K}^{\an}(+1)$ and hence $\jcal_n = \operatorname{Sym}^{n-1}(\jcal_2) = \operatorname{Sym}^{n-1}(\kcal^{\an}(+1)) = \mathrm{Log}^{\vee}_{n-1}(n-1)$. The main result of this section is that the constructions of the logarithm of the monodromy action due to Ayoub and Beilinson coincide. We begin with the following.
\begin{prop} \label{prop: isomorphisms of sp on geometric complexes}
    There are isomorphisms of specialization systems $\Upsilon_f^{\anor} \simeq \operatorname{log}_f^{\an} \simeq \Upsilon_f^{\an} \colon \dcat_{\geo}(X_{\eta}) \longrightarrow \dcat_{\geo}(X_{\sigma})$. In particular, there is a monodromy sequence
    \begin{equation*}
        i^*j_* \longrightarrow \Upsilon_f^{\an} \overset{N}{\longrightarrow} \Upsilon_f^{\an}(-1) \longrightarrow +1
    \end{equation*}
    where $N$ is nilpotent on $\dcat^b_{\geo}(-)$.
\end{prop}

 \begin{proof}
The first isomorphism $\Upsilon_f^{\anor} \simeq \operatorname{log}_f^{\an}$ can be verified similarly as in \cite[Th\'eorème 11.14]{ayoub-2014}: by \cite{ayoub-thesis-2}, one has a morphism $\kcal^{\an} \longrightarrow \uscr^{\un}$ and this results in a morphism of specialization system $\log^{\an} \longrightarrow \Upsilon_f^{\anor}$. By \ref{cor: criterion for iso of sp} and the know values of both $\Upsilon_f^{\anor},\log_f^{\an}$ on local-crossing model with one branch, we see that this is indeed an isomorphism of specialization systems. Concern the second one, let us first prove that $\log_f^{\an}(M) \in \perv_{\geo}(X_{\sigma})$ is perverse provided that $M \in \perv_{\geo}(X_{\sigma})$. By looking at the long exact sequence of the monodromy sequence 
\begin{equation*}
    \log_f^{\an}[-1] \longrightarrow \log_f^{\an}[-1](-1) \longrightarrow i^*j_* \longrightarrow +1
\end{equation*}
and the fact that $\chi_f$ concentrates on perverse degrees $0,-1$ (by proposition \ref{prop: perversity of operations on geometric perverse sheaves}), we see that $\log_f^{\an}[-1]$ concentrates on perverse degree $-1,0,1$. For $i \in \left \{-1,1 \right \}$, the map $\phnor^i(\log_f^{\an}[-1]) \longrightarrow \phnor^i(\log_f^{\an}[-1](-1))$ is either injective or surjective and nilpotent and hence $\phnor^i(\log_f^{\an}[-1])=0$. Consequently, $\log_f^{\an}[-1] = \phnor^0(\log_f[-1])$ is perverse. Now we apply a similar version of \cite[Lemme 11.19]{ayoub-2014} to see that there is a morphism of triangles
\begin{equation*}
\begin{tikzcd}[sep=large]
    i^*j_*(M \otimes f_{\eta}^*\mathrm{Log}^{\vee}_{m-1}) \arrow[r] \arrow[d] & \operatorname{log}_f^{\an}(M) \arrow[r,"N_m"] \arrow[d] &  \operatorname{log}_f^{\an}(M)(-m) \arrow[r] \arrow[d] & +1 \\
i^*j_*(N \otimes f_{\eta}^*\mathrm{Log}^{\vee}_{m-1}) \arrow[r] & \operatorname{log}_f^{\an}(N) \arrow[r,"N_m"] &  \operatorname{log}_f^{\an}(N)(-m) \arrow[r] & +1
    \end{tikzcd}
\end{equation*}
with $N_m$ being nilpotent and the sequence canonically splits for sufficiently large $m \geq 0$ thanks to lemma \ref{lem: nilpotency of monodromy operator}. In particular, there is an isomorphism 
\begin{equation*}
    i^*j_*(M \otimes f_{\eta}^*(\mathrm{Log}_{m-1}^{\vee})) = \log_f^{\an}(M) \oplus \log_f^{\an}(M)[-1](-m).
\end{equation*}
From the long exact sequence of perverse cohomology, we see that there is a commutative square with isomorphisms horizontally 
\begin{equation*}
\begin{tikzcd}[sep=large]
    \phnor^{-1}(i^*j_*(M \otimes f_{\eta}^*(\mathrm{Log}^{\vee}_{m-1}))) \arrow[r,"\sim"] \arrow[d] & \operatorname{log}_f^{\an}(M) \arrow[d] \\ 
\phnor^{-1}(i^*j_*(N \otimes f_{\eta}^*(\mathrm{Log}^{\vee}_{m-1}))) \arrow[r,"\sim"] & \operatorname{log}_f^{\an}(N).
    \end{tikzcd}
\end{equation*}
This coincides with the usual unipotent nearby functor and we see that $\log_f^{\an} \simeq \Upsilon_f^{\an}$ functorially. It remains to prove that this is a morphism of specialization systems, namely, to construct a natural transformation $\log_f^{\an} \overset{\sim}{\longrightarrow} \Upsilon_f^{\an}$ commuting with base change transformations $\alpha_g$. It suffices to separate two cases where $g$ is smooth or $g$ is a closed immersion. Let us treat the latter case first, the first diagram has all morphisms being morphisms of specialization systems, hence it induces the left commutative diagram (functorial in $M$, as showed above)
\begin{equation*}
\begin{tikzcd}[sep=large]
    \phnor^{-1}(\chi_f(g_{\eta,*}(M) \otimes f_{\eta}^*(\mathrm{Log}^{\vee}_{m-1}))) \arrow[r,"\sim"] \arrow[d] & \operatorname{log}_f^{\an}(g_{\eta,*}(M)) \arrow[d] \\ 
g_{\sigma,*}\phnor^{-1}(\chi_{f \circ g}(M \otimes f_{\eta}^*(\mathrm{Log}^{\vee}_{m-1}))) \arrow[r,"\sim"] & g_{\sigma,*}\operatorname{log}_{f \circ g}^{\an}(M)
    \end{tikzcd} \ \ \ \ \begin{tikzcd}[sep=large]
    \Upsilon_f^{\an}g_{\eta,*}(M) \arrow[r,"\sim"] \arrow[d] & \operatorname{log}_f^{\an}(g_{\eta,*}(M)) \arrow[d] \\ 
g_{\sigma,*}\Upsilon_{f \circ g}^{\an}(M) \arrow[r,"\sim"] & g_{\sigma,*}\operatorname{log}_{f \circ g}^{\an}(M)
    \end{tikzcd}
\end{equation*}
and hence by taking derived functors, we obtain the right commutative diagram. By taking adjoint of the right diagram, we obtain 
\begin{equation*}
    \begin{tikzcd}[sep=large]
        g_{\sigma}^*\Upsilon^{\an}_f \arrow[r] \arrow[d] &  g_{\sigma}^*\log^{\an}_f  \arrow[d]  \\ 
        \Upsilon_{f \circ g}^{\an} g_{\eta}^* \arrow[r] & \log^{\an}_{f \circ g} g_{\eta}^*.
    \end{tikzcd}
\end{equation*}
If $g$ is smooth of relative dimension $d$, we repeat the argument with $g^{\dagger}$. 
 \end{proof}

 \begin{cor}
     Modulo proposition \ref{prop: logarithm of nearby cycles}, there is an isomorphism of specialization systems $\Psi_f^{\anor} \simeq \Psi_f^{\an}$ on $\dcat_{\geo}(-)$. 
 \end{cor}
\subsection{Comparison of monodromy operators}
It remains to show that the operator $N \colon \Upsilon_f^{\an} \longrightarrow \Upsilon_f^{\an}(-1)$ is indeed the usual monodromy operator. We begin by some prerequisites, which already appears in the \'etale-motivic setting in \cite{ayoub-2014}. 
\begin{lem} \label{lem: monodromy matrix}
    Let $\sp$ be either $\Psi_f^{\an},\Upsilon_f^{\an}$, there is a canonical $\pi_1(\gm_{m,k})$-equivariant isomorphism
    \begin{equation*}
        \sp_f(M \otimes f_{\eta}^*\mathrm{Log}_{m-1}^{\vee}) \simeq \sp_f(M) \otimes_{\mathbb{Q}}  \left( \mathbb{Q} \oplus \mathbb{Q}(-1) \oplus \cdots \oplus \mathbb{Q}(1-m) \right)
    \end{equation*}
    where on the right hand side, $\Psi_f^{\an}(M)$ is acted on by the monodromy action and $\bigoplus_{i=0}^{m-1}\mathbb{Q}(-i)$ is acted on by the matrix
     \begin{equation*}
         \lnor_{m-1} =    \begin{pmatrix}
1 & -\binom{n}{1} & -\binom{n}{2}  & \cdots  &-\binom{n}{n} \\
 & 1 & -\binom{n-1}{1} & \cdots & -\binom{n-1}{n-1} \\
 &  & \ddots &  &  \\
 &  &  & 1 & -1 \\
0 &  &  &  & 1 \\
\end{pmatrix} 
\end{equation*}
\end{lem}
\begin{proof}
    By the Kunneth formula, one has that $\sp_f(M \otimes f_{\eta}^*\mathrm{Log}_{m-1}^{\vee}) \simeq \sp_f(M) \otimes f_{\sigma}^*\sp_{\id}(\mathrm{Log}_{m-1}^{\vee})$ so it suffices to compute the monodromy $\sp_{\id}(\mathrm{Log}_{m-1}^{\vee}) = \operatorname{Sym}^n(\sp_{\id}(\kcal^{\an}))$. The result then follows from lemma \ref{lem: monodromy of Kummer motive} since $\sp_{\id}(\kcal^{\an}) \simeq \sp_{\id}(\mathds{1}) \oplus \sp_{\id}(\mathds{1}(-1)) \simeq \mathds{1} \oplus \mathds{1}(-1)$ has monodromy given by the matrix $L_2$.
\end{proof}
\begin{prop}
There is a natural commutative diagram 
    \begin{equation*}
    \begin{tikzcd}[sep=large]
        \log_f^{\an} \arrow[r,"N=N_1"] \arrow[d,"\sim"] & \log^{\an}_f(-1) \arrow[d,"\sim"] \\ 
        \Upsilon_f^{\an} \arrow[r,"N^{\an}"] & \Upsilon_f^{\an}(-1).
    \end{tikzcd}
\end{equation*}
\end{prop}
\begin{proof}
    To prove the proposition, it is enough to show that $T = \exp(g_1N) = \sum_{m=0}^{\infty} (g_1 N)^{\circ m}/m!$ (this makes sense since $N$ is nilpotent). In other words, we have to show that there is a commutative diagram 
    \begin{equation*}
    \begin{tikzcd}[sep=large]
        \Upsilon_f^{\an}(M) \arrow[r,"\exp(g_1N)"] \arrow[d,"\operatorname{can}"] & \Upsilon_f^{\an}(M)(-1) \arrow[d,"\operatorname{can}"] \\ 
        \Psi_f^{\an}(M) \arrow[r,"T"] & \Psi_f^{\an}(M)(-1).
    \end{tikzcd}
\end{equation*}
for any $M \in \dcat_{\geo}^b(X_{\eta})$ with $\operatorname{can} \colon \Upsilon_f^{\an} \longhookrightarrow \Psi_f^{\an}$ the canonical morphism. The proof below is essentially due to Ayoub in the \'etale-motivic setting (see \cite[Théorème 11.7]{ayoub-2014}). The morphism
    \begin{equation*}
        \chi_f(M \otimes f_{\eta}^*\mathrm{Log}^{\vee}_{m-1}) \longrightarrow \Psi_f^{\an}(M \otimes f_{\eta}^*\mathrm{Log}^{\vee}_{m-1})
    \end{equation*}
    is $\pi_1(\gm_{m,k})$-equivariant where the action on the left is trivial and the right is the monodromy action. On the other hand, there is a canonical $\pi_1(\gm_{m,k})$-equivariant isomorphism 
    \begin{equation*}
        \Psi_f^{\an}(M \otimes f_{\eta}^*\mathrm{Log}^{\vee}_{m-1}) \simeq \Psi_f^{\an}(M) \otimes_{\mathbb{Q}} \left( \mathbb{Q} \oplus \mathbb{Q}(-1) \oplus \cdots \oplus \mathbb{Q}(1-m) \right),
    \end{equation*}
    thanks to lemma \ref{lem: monodromy matrix}. After lemma \ref{lem: nilpotency of monodromy operator}, we have an isomorphism $\log_f^{\an}(M) \oplus \log_f^{\an}(M)(-m)[-1] \simeq \chi_f(M \otimes f_{\eta}^*\mathrm{Log}^{\vee}_{m-1})$, in which $\pi_1(\gm_{m,k})$ acts trivially on $\log_f^{\an}(M)$ as it acts trivially on $\chi_f(M \otimes f_{\eta}^*\mathrm{Log}^{\vee}_{m-1})$. In particular, there are $\pi_1(\gm_{m,k})$-equivariant morphisms 
    \begin{equation*}
        \varphi_r \colon \Upsilon_f^{\an}(M)  \longrightarrow \Psi_f^{\an}(M) \otimes_{\mathbb{Q}} \left( \mathbb{Q} \oplus \mathbb{Q}(-1) \oplus \cdots \oplus \mathbb{Q}(1-m) \right) \overset{\textnormal{projection}}{\longrightarrow} \Psi_f^{\an}(M)(-r).
    \end{equation*}
These morphisms equal to the composition
\begin{equation*}
\begin{split}
    \log_f^{\an}(M) \longrightarrow \chi_f(M \otimes f_{\eta}^*\mathrm{Log}_{m-1}^{\vee}) & \longrightarrow \log_f^{\an}(M \otimes f_{\eta}^*\mathrm{Log}_{m-1}^{\vee}) \\ 
    & \simeq \bigoplus_{i=0}^{m-1} \log_f^{\an}(M)(-i) \longrightarrow \bigoplus_{i=0}^{m-1}\Psi^{\an}_f(M)(-i) \longrightarrow \Psi_f^{\an}(M)(-r).
    \end{split}
\end{equation*}
By an analogue of \cite[Proposition 11.20]{ayoub-2014}, the morphism above equals to
    \begin{equation*}
        \log_f^{\an}(M) \overset{N_r}{\longrightarrow} \log_f^{\an}(M)(-r) \longrightarrow \Psi_f^{\an}(M)(-r).
    \end{equation*}
Thus if we denote by $\operatorname{can} \colon\Upsilon_f^{\an} \longhookrightarrow \Psi_f^{\an}$ the canonical morphism then we have that $\operatorname{can} \circ N_r = \varphi_r$. Since everything is $\pi_1(\gm_{m,k})$-equivariant, the following compositions equal
\begin{align*}
        \Upsilon_f^{\an}(M) \longrightarrow \bigoplus_{i=0}^{m-1} \Psi_f^{\an}(M)(-i) \overset{\textnormal{projection}}{\longrightarrow} \Psi_f^{\an}(M) \\ 
        \Upsilon_f^{\an}(M) \longrightarrow \bigoplus_{i=0}^{m-1} \Psi_f^{\an}(M)(-i) \overset{T \otimes \lnor_{m-1}}{\longrightarrow } \bigoplus_{i=0}^{m-1} \Psi_f^{\an}(M)(-i) \overset{\textnormal{projection}}{\longrightarrow} \Psi_f^{\an}(M)
    \end{align*}
By definition and lemma \ref{lem: monodromy matrix},  $\operatorname{can} = \sum_{i=0}^{m-1} (T \otimes \lnor_{m-1}^{-i}) \circ \varphi_i$. Hence, using all of these above and the fact that $m_!N_m = N^{\circ m}$ in lemma \ref{lem: nilpotency of monodromy operator}, we see that
\begin{equation*}
    \operatorname{can} =  T \circ \operatorname{can} \circ \left(\sum_{i=0}^{m-1} g_1^{-i} N_i \right) = T \circ \operatorname{can} \circ \left(\sum_{i=0}^{m-1} \frac{(g_1^{-1} N)^i}{i!} \right) = T \circ \operatorname{can} \circ \operatorname{exp}(g_1^{-1}N)
\end{equation*}
and since $\operatorname{exp}(g_1^{-1}N) = \operatorname{exp}(g_1N)^{-1}$, we finish the proof. 
\end{proof}

\begin{cor} \label{cor: unipotent nearby cycles under betti realizations}
Let $f \colon X \longrightarrow \mathbb{A}_k^1$ be a morphism of $k$-varieties, there is a natural isomorphism 
\begin{equation*}
    \betti^*_{X_{\sigma}} \circ \Upsilon_f^{\anor} \overset{\simeq}{\longrightarrow}  \Upsilon_f^{\an} \circ \betti^*_{X_{\eta}}, 
\end{equation*}
such that the monodromy sequences are compatible. In other words, there is a commutative diagram of the form 
\begin{equation*}
    \begin{tikzcd}[sep=large]
        \betti^*_{X_{\sigma}} \circ i^*j_* \arrow[r] \arrow[d,"\sim"] &  \betti^*_{X_{\sigma}} \circ \Upsilon_f^{\anor} \arrow[r,"N"] \arrow[d,"\sim"] &  \betti^*_{X_{\sigma}} \circ \Upsilon_f^{\anor}(-1) \arrow[r] \arrow[d,"\sim"] & +1 \\ 
         i^*j_* \circ  \betti^*_{X_{\eta}} \arrow[r] & \Upsilon_f^{\an} \circ \betti^*_{X_{\eta}} \arrow[r,"\log(T-1)"] & \Upsilon_f^{\an} \circ \betti^*_{X_{\eta}}(-1) \arrow[r] & +1
    \end{tikzcd}
\end{equation*}
\end{cor}

\begin{proof}
Given the previous lemma, this is obvious as both sides are defined in a compatible, functorial way. 
\end{proof}

\begin{rmk}
    Historically, corollary \ref{cor: unipotent nearby cycles under betti realizations} is proven in \cite[Lemma 3.13]{florian+morel-2019}. However, in \cite[Lemma 3.13]{florian+morel-2019}, the authors do not verify $\Upsilon_f^{\an} \simeq \Upsilon_f^{\anor}$ as well as the compatibility of monodromy sequences. 
\end{rmk}

The same line of arguments yields the following, which is used in the next section to define external tensor products of unipotent nearby functors on perverse Nori motives. 
\begin{cor} \label{cor: compatibility of external products, unipotent nearby functors and betti realizations}
    Let $f \colon X \longrightarrow \mathbb{A}_k^1, g \colon X \longrightarrow \mathbb{A}_k^1$ be morphisms of $k$-varieties, there is a commutative diagram
    \begin{equation*}
        \begin{tikzcd}[sep=large]
            \betti^*(\Upsilon^{\anor}_f(-)) \boxtimes \betti^*(\Upsilon_g^{\anor}(-)) \arrow[r,"\sim"] \arrow[d,"\sim"] &  \betti^*(\Upsilon^{\anor}_f(-) \boxtimes \Upsilon_g^{\anor}(-)) \arrow[r,"\sim"] & \betti^*(\Upsilon_{f \times g}^{\anor}((-) \boxtimes (-)))  \arrow[d,"\sim"] \\ 
             \Upsilon_f^{\an}(\betti^*(-)) \boxtimes \Upsilon_g^{\an}(\betti^*(-)) \arrow[r,"\sim"] & \Upsilon_{f \times g}^{\an}(\betti^*(-) \boxtimes \betti^*(-)) \arrow[r,"\sim"] & \Upsilon_{f \times g}^{\an}(\betti^*((-) \boxtimes (-))).
        \end{tikzcd}
    \end{equation*}
\end{cor}

\section{Nori-theoretic nearby functors on Nori motives}

In \cite{ayoub-2010}, Ayoub shows that motivic nearby functors and Betti nearby functors are compatible under the Betti realization. By the universal property of Nori motives and the perverse $t$-exactness of $\Psi^{\an}_f[-1]$, we deduce the existence of a functor 
\begin{equation*}
    \Psi_f^{\univ}[-1] \colon \mscr\perv(X_{\eta}) \longrightarrow \mscr\perv(X_{\sigma}). 
\end{equation*}
In particular, by taking the derived functor and shifting, we obtain a functor
\begin{equation*}
    \Psi_f^{\univ}[-1] \colon \dn^b(X_{\eta}) \longrightarrow \dn^b(X_{\sigma}). 
\end{equation*}
At this point, there are two natural questions arising:
\begin{enumerate}
    \item Does this functor coincide with Ayoub's functor, namely, do we have an isomorphism $\Psi_f^{\univ} \simeq \Psi_f^{\anor}$? 
    \item Does the Beilinson's unipotent nearby functor $\Upsilon_f^{\an}$ compatible with Ayoub's unipotent nearby functor $\Upsilon_f^{\mot}$ under the Betti realization? If yes, this results in a universal functor $\Upsilon_f^{\univ}$ on Nori motives and does we have again $\Upsilon_f^{\univ} \simeq \Psi_f^{\anor}$ as in the case of the full nearby functor?
    \item If the answer of the second question is yes, can we compute the fiber of $\Upsilon_f^{\univ} \longrightarrow \Upsilon_f^{\univ}(-1)$? In other words, can we expect a monodromy triangle in the setting of Nori motives?
\end{enumerate}
We provide affirmative answers to all of these questions in this section. 
\subsection{The Nori-theoretic nearby functors}

As discussed in the beginning of this section, we have two functors 
\begin{equation*}
    \Psi_f^{\univ},\Upsilon_f^{\univ} \colon \dn(X_{\eta}) \longrightarrow \dn(X_{\sigma}).
\end{equation*}
It is however unclear how these functors form specialization systems. Of course it should be the case because we will show that they coincide with Ayoub's functors. However, we stress to the fact that we want a comparison result for specialization systems, not just of functors. Thus, at first we have to check that universal functors form specialization systems. 
\begin{prop} \label{prop: universal functors are sp}
    Proposition \ref{prop: functoriality of nearby cycles} holds true for $\Psi_f^{\univ},\Upsilon_f^{\univ}$ over quasi-projective varieties. 
\end{prop}
\begin{proof}
Let us focus on $\Psi_f^{\univ},\Upsilon^{\univ}$ and only on $\Upsilon_f^{\univ}$ since the same applies for $\Psi_f^{\univ}$. For $(1),(2)$, we separate two cases: $g$ is smooth and $g$ is a closed immersion as in the proof of proposition \ref{prop: isomorphisms of sp on geometric complexes}. If $g$ is smooth, then the compatibility of $\Upsilon_f^{\anor}$ (on \'etale motives) and $\Upsilon_f^{\an}$ (on complexes of geometric origin) shows that there is a natural transformation
 \begin{equation*}
     g_{\sigma}^{\dagger}\Upsilon_f^{\univ}[-1] \longrightarrow \Upsilon_{f \circ g}^{\univ}[-1]g_{\eta}^{\dagger}.
 \end{equation*}
 If $g$ is a closed immersion, we switch to $\beta_g$ and repeat the argument. The case of duality $(3)$ follows by the same argument. Since $(-) \otimes (-) = \Delta^*((-) \boxtimes (-))$, where $\Delta$ is the diagonal morphism, in order to construct 
 \begin{equation*}
     \Upsilon_f^{\univ}(-) \otimes \Upsilon_f^{\univ}(-) \longrightarrow \Upsilon_f^{\univ}((-) \otimes (-))
 \end{equation*}
 we just have to construct a morphism
 \begin{equation*}
     \Upsilon^{\univ}_f(-) \boxtimes \Upsilon_g^{\univ}(-) \overset{\sim}{\longrightarrow} \Upsilon_{f \times g}^{\univ}((-) \boxtimes (-))
 \end{equation*}
 and then set $f=g$, apply $\Delta^*$ and use the base change morphism. The external product $\boxtimes$ is perverse $t$-exact thanks to the work \cite{terenzi-2024} so it suffices to construct an isomorphism 
  \begin{equation*}
     \Upsilon^{\univ}_f(-)[-1] \boxtimes \Upsilon_g^{\univ}(-)[-1] \overset{\sim}{\longrightarrow} \Upsilon_{f \times g}^{\univ}((-) \boxtimes (-))[-1]
 \end{equation*}
 on $\mscr\perv((-)_{\eta}) \times \mscr\perv((-)_{\eta})$. This follows from the universal proerty of perverse Nori motives and corollary \ref{cor: compatibility of external products, unipotent nearby functors and betti realizations}.
\end{proof}

\begin{prop} 
   Proposition \ref{prop: computation of nearby cycles} holds true for $\Psi_f^{\univ},\Upsilon_f^{\univ}$ over quasi-projective varieties and there is a canonical isomorphism of specialization systems
   \begin{equation*}
        \Psi_f^{\univ} \simeq \colim_{n \in \mathbb{N}} \ \Upsilon_{f_n}^{\univ}(e_n)_{\eta}^*.
   \end{equation*}
   Moreover, for any $M \in \dn^b(X_{\eta})$, there exists an integer $n_0=n_o(M)$ such that $\Upsilon_{f_n}^{\univ}(e_n)_{\eta}^*(M) \simeq \Psi_f^{\univ}$ for any $n$ divisible by $n_0$.
\end{prop}
\begin{proof}
All the properties follow from known cases for \'etale motives and complexes of geometric origin and the fact that all the involved functors are perverse $t$-exact on complexes of geometric origin. Let us prove $\Psi_f^{\univ} \simeq \colim_{n \in \mathbb{N}} \ \Upsilon_{f_n}^{\univ}(e_n)_{\eta}^*$. We observe that there are canonical morphisms of specialization systems $\Upsilon_{f_n}^{\univ}(e_n)^{\dagger}_{\eta} \longrightarrow \Psi_f^{\univ}$. Indeed, it suffices to construct $\Upsilon_{f_n}^{\univ}[-1](e_n)^{\dagger}_{\eta} \longrightarrow \Psi_f^{\univ}[-1]$ on $\mscr\perv(X_{\sigma})$ since $(e_n)_{\eta}^{\dagger}$ is perverse $t$-exact. Such transformations are available in both motivic and analytic settings and they are compatible and hence yields the one on universal factorizations. These assembles to a morphism $\colim_{n \in \mathbb{N}} \ \Upsilon_{f_n}^{\univ}(e_n)_{\eta}^* \longrightarrow \Psi_f^{\univ}$. Two sides commute with direct sums so we just have to compare them on $\dn^b(X_{\eta})$. By an analogue of proposition \ref{prop: logarithm of nearby cycles}, for any $M \in \dn^b(X_{\eta})$, there exists an integer $n_0$ such that 
    \begin{equation*}
        \Upsilon_{f_n}^{\univ}(e_n)_{\eta}^*(M) \overset{\sim}{\longrightarrow} \colim_{n \in \mathbb{N}} \ \Upsilon_{f_n}^{\univ}(e_n)_{\eta}^* \overset{\sim}{\longrightarrow} \Psi_f^{\univ}
    \end{equation*}
    so we win.  
\end{proof}
\subsection{The universal functor coincide with Ayoub functor}
Our goal here is to prove that these "universal functors" coincide with "Ayoub functors", namely, $\Upsilon_f^{\univ} = \Upsilon_f^{\anor}$ and $\Psi_f^{\univ} = \Psi_f^{\anor}$. It suffices to compare them on $\dn^b(X_{\eta})$ since Ayoub's functors preserve compact objects thanks to proposition \ref{prop: compactness of Ayoub's functors}.
\begin{lem}
    The functor $\Psi_f^{\anor}[-1] \colon \dn^b(X_{\eta}) \longrightarrow \dn^b(X_{\sigma})$ is perverse $t$-exact.
\end{lem}
\begin{proof}
    The functor $\dn^b(-) \longrightarrow \dcat^b_{\geo}(-)$ is conservative and perverse $t$-exact so it suffices to prove the same result for $\dcat^b_{\geo}(-)$. By proposition \ref{prop: logarithm of nearby cycles}, we reduce to the perverse $t$-exactness of $\Upsilon_f^{\anor}[-1]$. This follows from corollary \ref{cor: criterion for compactness} and proposition \ref{prop: restrictions of nearby functors on geometric complexes}.
\end{proof}
\begin{theorem} \label{thm: nearby functors on Nori motives coincide}
There exists a natural isomorphism of specialization systems (in the sense of \cite[Definition 3.1.1]{ayoub-thesis-2})
\begin{align*}
    \Psi^{\univ}_f \simeq \Psi_f^{\anor} \colon \dn^b(X_{\eta}) \longrightarrow \dn^b(X_{\sigma}) \\ 
    \Upsilon^{\univ}_f \simeq \Upsilon_f^{\anor} \colon \dn^b(X_{\eta}) \longrightarrow \dn^b(X_{\sigma}).
\end{align*}
In particular, there is a chain of specialization systems under various realization functors
\begin{equation*}
    \begin{tikzcd}[column sep=large, row sep = 0.2]
        \daetct(-) \arrow[rr,"\betti^*",bend left =  30] \arrow[r,"\operatorname{Nri}^*"] & \dn^b(-) \arrow[r] & \dcatct^b(-) \\ 
        \Psi_f^{\anor},\Upsilon_f^{\anor} \arrow[r] & \Psi_f^{\univ},\Upsilon_f^{\univ} \arrow[r] & \Psi_f^{\an},\Upsilon_f^{\an}.
        \end{tikzcd}
    \end{equation*}
\end{theorem}

\begin{proof}
First we see that the functors $\Psi_f^{\anor}[-1],\Upsilon_f^{\anor}[-1] \colon \dn^b(X_{\eta}) \longrightarrow \dn^b(X_{\sigma})$ are perverse $t$-exact since these functors are compatible with the Betti realizations and Betti realizations are conservative and perverse $t$-exact. Consequently, it suffices to prove that 
\begin{equation*} 
\Psi^{\univ}_f[-1] \simeq \Psi_f^{\anor}[-1] \colon \mscr\perv(X_{\eta}) \longrightarrow \mscr\perv(X_{\sigma}).
\end{equation*} 
Since Ayoub's nearby functors are compatible under realizations, there is a commutative diagram 
\begin{equation*}
    \begin{tikzcd}[row sep=large, column sep = scriptsize]
    & &  \mscr\perv(X_{\eta}) \arrow[dd,dashed] \arrow[rr] & & \perv(X_{\eta}) \arrow[dd,"{}^p\Psi_f^{\an}"] \\ 
        \daetct(X_{\eta}) \arrow[rru,color=blue,bend left = 10] \arrow[r]  \arrow[dd,"{}^p\Psi_f^{\anor}",swap] & \dn^b(X_{\eta})  & &\dcatct^b(X_{\eta}) \arrow[ru,"\phnor^0"] & \\
        & & \mscr\perv(X_{\sigma}) \arrow[rr,dashed] & &  \perv(X_{\sigma}) \\ 
        \daetct(X_{\sigma}) \arrow[r]  \arrow[rru,color=blue,bend left = 10,dashed]  \arrow[rrr,bend right = 10,"\betti^*_{X_{\sigma}}",swap]  & \dn^b(X_{\sigma}) \arrow[rr] &   & \dcatct^b(X_{\sigma}) \arrow[ru,"\phnor^0"]  & \arrow[from=2-2, to=1-3, dashed] \arrow[from=4-2, to=3-3,dashed] \arrow[from=2-1,to=2-4,bend left = 10,"\betti^*_{X_{\eta}}",pos=0.3,crossing over] \arrow[from=2-4,to=4-4,"{}^p\Psi_f^{\an}",crossing over,pos = 0.3] \arrow[from=2-2,to=2-4,crossing over]  \arrow[from=2-2,to=4-2,"{}^p\Psi_f^{\anor}",crossing over].
    \end{tikzcd}
\end{equation*}
The two blue arrows are precisely the universal cohomology in the universal abelian factorization (see \cite[Corollary 4.20]{tubach-2025}). This results in the following commutative diagram by taking the zero-th perverse cohomology 
\begin{equation*}
    \begin{tikzcd}[sep=large]
        \daetct(X_{\eta}) \arrow[r] \arrow[d,"{}^p\Psi_f^{\anor}",swap] & \mscr\perv(X_{\eta})  \arrow[r] \arrow[d,"{}^p\Psi_f^{\anor}"] & \perv(X_{\eta}) \arrow[d,"{}^p\Psi_f^{\an}"] \\
        \daetct(X_{\sigma}) \arrow[r] & \mscr\perv(X_{\sigma}) \arrow[r]  & \perv(X_{\sigma}).
    \end{tikzcd}
\end{equation*}
However, up to a natural isomorphism, there exists a unique such middle \textit{exact} arrow making the diagram commutative, namely, $\Psi_f^{\univ}$ (see \cite[Proposition 2.5]{terenzi-2024}); in other words, $\Psi_f^{\univ} =\Psi_f^{\anor}$. Note that this does not imply immediately that this is an isomorphism of specialization systems but we can repeat the argument to see that there is a commutative diagram
\begin{equation*}
    \begin{tikzcd}[sep=large]
        \daetct(X_{\eta}) \arrow[r] \arrow[d,"g_{\sigma}^{\dagger} {}^p\Psi_f^{\anor}",swap] & \mscr\perv(X_{\eta})  \arrow[r] \arrow[d,"g_{\sigma}^{\dagger} {}^p\Psi_f^{\anor}"] & \perv(X_{\eta}) \arrow[d,"g_{\sigma}^{\dagger}{}^p\Psi_f^{\an}"] \\
        \daetct(Y_{\sigma}) \arrow[r] & \mscr\perv(X_{\sigma}) \arrow[r]  & \perv(X_{\sigma}).
    \end{tikzcd}
\end{equation*} 
provided that $g \colon Y \longrightarrow X$ is a smooth morphism. By using \cite[Proposition 2.5]{terenzi-2024} again, we obtain that $g_{}$
$g^{\dagger}_{\sigma}\Psi_f^{\univ}[-1] \simeq \Psi_{f \circ g}^{\univ}[-1]g_{\eta}^{\dagger} = \Psi_{f \circ g}^{\anor}[-1] g_{\eta}^{\dagger} \simeq g_{\sigma}^{\dagger}\Psi_f^{\anor}[-1]$ if $g$ is smooth and similar isomorphisms if $g$ is a closed immersion and then taking adjoint as in the proof of proposition \ref{prop: isomorphisms of sp on geometric complexes}. 

\end{proof}

\subsection{The monodromy sequence on Nori motives} Thanks to corollary \ref{cor: unipotent nearby cycles under betti realizations}, the following definition makes sense
\begin{defn}
   The \textit{Nori-motivic monodromy operator}
 \begin{equation*}
     N^{\univ} \colon \Upsilon_f^{\univ}[-1] \longrightarrow \Upsilon_f^{\univ}(-1)[-1]
 \end{equation*}
 is transformation obtained by using the property \textbf{P2} in \cite[Section 2]{florian+morel-2019}. 
\end{defn}
\begin{cor}
    The operator $N^{\univ}$ is nilpotent on $\dn^b(X_{\eta})$. 
\end{cor}
\begin{proof}
    Indeed, it suffices to prove that $N^{\univ}$ is nilpotent on $\mscr\perv(X_{\eta})$ since the standard $t$-structure on $\dn^b(X_{\eta})$ is bounded. This follows from the fact that both monodromy operators on \'etale motives and constructible complexes of geometric origin are nilpotent. 
\end{proof}
It is however unclear to us that there should exist a monodromy sequence
\begin{equation*}
    i^*j_* \longrightarrow \Upsilon_f^{\univ} \overset{N^{\univ}}{\longrightarrow} \Upsilon_f^{\univ}(-1) \longrightarrow +1
\end{equation*}
on Nori motives. To do this, it suffices to show the existence of a monodromy sequence for $\Upsilon_f^{\anor}$ thanks to \ref{thm: nearby functors on Nori motives coincide}. The idea is to use the Nori-theoretic logarithmic specialization system. The Kummer extension for Nori motives
\begin{equation*}
    \mathds{1}_{\gm_m} \longrightarrow \kcal^{\nori} \longrightarrow \mathds{1}_{\gm_m}(-1) \longrightarrow +1.
\end{equation*}
(we do not claim that this cofiber sequence is uniquely determined as we cannot compute the extension group for Nori motives) gives rise to a monodromy sequence
\begin{equation*}
    i^*j_* \longrightarrow \log_f^{\nori} \overset{N}{\longrightarrow} \log_f^{\nori}(-1) \longrightarrow +1
\end{equation*}
and now we have that
\begin{prop}
There are isomorphisms of specialization systems $\Upsilon^{\univ} \simeq \Upsilon^{\anor} \simeq \log_f^{\nori} \colon \dn^b(X_{\eta}) \longrightarrow \dn^b(X_{\sigma})$ and there is a commutative diagram
\begin{equation*}
    \begin{tikzcd}[sep=large]
        \Upsilon_f^{\univ} \arrow[r,"N^{\univ}"] \arrow[d,"\sim"] & \Upsilon^{\univ}_f(-1) \arrow[d,"\sim"] \\ 
        \log_f^{\nori} \arrow[r,"N"] & \log_f^{\nori}(-1).
    \end{tikzcd}
\end{equation*}
In particular, there exists a monodromy sequence
\begin{equation*}
    i^*j_* \longrightarrow \Upsilon_f^{\univ} \overset{N^{\univ}}{\longrightarrow} \Upsilon_f^{\univ}(-1) \longrightarrow +1.
\end{equation*}
\end{prop}

\begin{proof}
The isomorphism $\Upsilon_f^{\univ} \simeq \Upsilon_f^{\anor}$ is theorem \ref{thm: nearby functors on Nori motives coincide}. Now the functor $\log_f^{\nori}$ preserves compact motives by corollary \ref{cor: log sp preserves compact motives}. As in the proof of proposition \ref{prop: isomorphisms of sp on geometric complexes}, there is a morphism of specialization systems $\log_f^{\nori} \longrightarrow \Upsilon_f^{\anor}$ and we just have to check their values on one-branch local-crossing models and we are done since they are images of corresponding functors on \'etale motives. As in the proof of, we obtain a commutative diagram
\begin{equation*}
    \begin{tikzcd}[column sep=scriptsize, row sep=large]
        \daetct(X_{\eta}) \arrow[rd,"\log_f^{\anor}\simeq \Upsilon_f^{\anor}"{name=bang1}] \arrow[rr,"\phnor^{\univ}"] \arrow[dd,equal] & & \mscr\perv(X_{\eta}) \arrow[dd,equal,dashed] \arrow[rr] \arrow[rd,"\Upsilon_f^{\univ}"]& & \perv(X_{\eta}) \arrow[dd,equal,dashed] \arrow[rd,"\Upsilon_f^{\an}"{name=bang3}] & \\ 
        & \daetct(X_{\sigma}) \arrow[rr,"\phnor^{\univ}",near start,crossing over]   & & \mscr\perv(X_{\sigma})  \arrow[rr,crossing over] & & \perv(X_{\eta}) \arrow[dd,equal] \\ 
        \daetct(X_{\eta}) \arrow[rd,"\log_f^{\anor}(-1) \simeq \Upsilon_f^{\anor}(-1)"{name=bang2},swap] \arrow[rr,"\phnor^{\univ}",near end,dashed]  & & \mscr\perv(X_{\eta}) \arrow[rr,dashed] \arrow[rd,"\Upsilon_f^{\univ}(-1)",dashed] & & \perv(X_{\eta}) \arrow[rd,"\Upsilon_f^{\an}(-1)"{name=bang4}] & \\ 
        & \daetct(X_{\sigma}) \arrow[rr,"\phnor^{\univ}"]  & & \mscr\perv(X_{\sigma}) \arrow[rr] & & \perv(X_{\eta}) \arrow[from=2-4, to=4-4,equal,crossing over] 
        \arrow[from=2-2, to=4-2,equal,crossing over] \arrow[Rightarrow, from=2-2, to=3-1, "N",shorten <=15pt, shorten >=15pt] \arrow[Rightarrow, from=2-4, to=3-3, "\theta",shorten <=15pt, shorten >=15pt]  \arrow[Rightarrow, from=2-6, to=3-5, "N",shorten <=17pt, shorten >=17pt]
    \end{tikzcd}
\end{equation*}
The rest of the argument is similar to the proof of theorem \ref{thm: nearby functors on Nori motives coincide}, in which we replace \cite[Proposition 2.5]{terenzi-2024} by \cite[Proposition 3.4]{terenzi-2024} (see also \cite[Section 2, Property \textbf{P2}]{florian+morel-2019}); namely, the transformation $\theta \simeq N^{\univ}$ is necessarily unique up to an isomorphism.
\end{proof}

\section{Application: The motivic integral identity}
In this section, we deduce the commutation between the Braden transformations and the nearby functors. In particular, the motivic integral identity of Kontsevich-Soibelman follows. 
\subsection{Commutation with Braden transformations}
Let $X$ be a $k$-variety endowed with a $\gm_{m,k}$-action. Follow \cite{drinfeld+gaitsgory-2014}\cite{richarz-2018}, we can define three new spaces. If $T$ is a $k$-variety, then we define the \textit{space of fixed points as}, the \textit{attractor}, the \textit{repeller}, respectively 
\begin{align*}
X^0(T) & = \left \{\text{equivariant morphisms} \ T \longrightarrow X_T \right \} \\ 
    X^+(T) & = \left \{\text{equivariant morphisms} \ (\mathbb{A}^1_T)^+ \longrightarrow X_T\right \} \\ 
    X^-(T) & = \left \{\text{equivariant morphisms} \ (\mathbb{A}^1_T)^- \longrightarrow X_T\right \}.
\end{align*}
(we note that these are algebraic spaces at first, but they are indeed schemes thanks to \cite{richarz-2018}). The structural morphisms $(\mathbb{A}^1_k)^{\pm} \longrightarrow \Spec(k)$ define morphisms $s^{\pm} \colon X^0 \longrightarrow X^{\pm}$. The zero sections $\Spec(k) \longrightarrow (\mathbb{A}_k^1)^{\pm}$ define morphisms $\pi^{\pm} \colon X^{\pm}  \longrightarrow X^0$ such that $\pi^{\pm} \circ s^{\pm} = \id$. Analogously, the inclusions $\gm_{m,S} \longhookrightarrow (\mathbb{A}_S^1)^{\pm}$ define morphisms $e^{\pm} \colon X^{\pm} \longrightarrow X$ such that $e^+ \circ s^+ = e^- \circ s^-$ is the inclusion functor $X^0 \longhookrightarrow X$. As in \cite{braden-2003}\cite{richarz-2018}\cite{drinfeld+gaitsgory-2014}\cite{drinfeld-2015}, there is the \textit{Braden transformation} 
\begin{equation*}
     (\pi^-)_*(e^-)^! \longrightarrow (\pi^+)_!(e^+)^*.
\end{equation*}
Let us say a motive $M \in \dn(X)$ is $\gm_m$-\textit{equivariant} if there is an isomorphism $a^*(M) \simeq p^*(M)$ in $\dn(\gm_{m,k} \times_k X)$, where $a,p \colon \gm_{m,k} \times X \longrightarrow X$ is the action and the projection, respectively. A variety $X$ (or more generally, an Artin stack) is said to \textit{$\et$-locally linearizable} if there exists an \'etale cover $(u_i \colon U_i \longrightarrow X)_{i\in I}$ consisting of $\gm_m$-equivariant \'etale morphism with $U_i$ affine. Braden's theorem asserts that if the $\gm_m$-action on $X$ is $\et$-locally linearizable, then the Braden transformation is an isomorphism on $\gm_m$-equivariant motive. 
\begin{prop}
Let $X$ be a $k$-variety equipped with a $\et$-locally linearizable $\gm_{m,k}$-action. Let $f \colon X \longrightarrow \mathbb{A}^1_k$ be a $\gm_{m,k}$-equivariant morphism where the target is equipped with the trivial $\gm_{m,k}$-action, then there exists a commutative diagram 
   \begin{equation*}
        \begin{tikzcd}[sep=large]
            (\pi^-_{\sigma})_*(e^-_{\sigma})^!\Psi_f^{\univ}(M) \arrow[d] & \Psi_{f^0}^{\univ}(\pi^-_{\eta})_*(e^-_{\eta})^!(M) \arrow[d] \arrow[l] \\ 
            (\pi^+_{\sigma})_!(e^+_{\sigma})^*\Psi_f^{\univ}(M) \arrow[r]  & \Psi_{f^0}^{\univ} (\pi_{\eta}^+)_!(e^+_{\eta})^*(M)
        \end{tikzcd}
    \end{equation*}
    natural in $M \in \dn(X_{\eta})$ and all arrows are isomorphisms provided that $M$ is $\gm_m$-equivariant (the statement also holds for $\Upsilon_f^{\univ}$).
\end{prop}

\begin{proof}
Under the isomorphism $\Psi_f^{\anor} \simeq \Psi_f^{\univ}$, the method in \cite{bang-2024} applies to this situation. The only subtle point is that in \cite{bang-2024}, the author uses the description of nearby functors in terms of algebraic derivators $\Psi_f^{\anor} = (p_{\Delta \times \mathbb{N}^{\times}})_{\#}i^*j_*(\theta_f)_*(\theta_f)^*(p_{\Delta \times \mathbb{N}^{\times}})^*$. However, it is unclear that whether one can define $\dn(-)$ for diagrams of varieties so we propose a different solution to this. This one also proves the compatibility of Braden transformations with tensor products, that, to our best knowledge, is not yet available in the literature. The method in \cite{bang-2024} applies to the standard specialization system $\chi_f = i^*j_*$ yielding a commutative diagram
\begin{equation*}
        \begin{tikzcd}[sep=large]
            (\pi^-_{\sigma})_*(e^-_{\sigma})^!\chi_f(M \otimes f_{\eta}^*(\uscr)) \arrow[d] & \chi_{f^0}(\pi^-_{\eta})_*(e^-_{\eta})^!(M \otimes f_{\eta}^*(\uscr)) \arrow[d] \arrow[l] \\ 
            (\pi^+_{\sigma})_!(e^+_{\sigma})^*\chi_f(M \otimes f_{\eta}^*(\uscr)) \arrow[r]  & \chi_{f^0} (\pi_{\eta}^+)_!(e^+_{\eta})^*(M \otimes f_{\eta}^*(\uscr))
        \end{tikzcd}
    \end{equation*}
so it suffices to show that Braden's transformations are compatible with tensor products in the sense that there is a commutative diagram of the form
\begin{equation*}
    \begin{tikzcd}[sep=large]
          (\pi_{\eta}^-)_*(e_{\eta}^-)^!(M)\otimes (f^0_{\eta})^*(\uscr) \arrow[d] \arrow[r] & (\pi^-_{\eta})_*(e^-_{\eta})^!(M \otimes f_{\eta}^*(\uscr)) \arrow[d]  \\ 
          (\pi^+_{\eta})_!(e^+_{\eta})^*(M)\otimes (f^0_{\eta})^*(\uscr)\arrow[r] & (\pi^+_{\eta})_!(e^+_{\eta})^*(M \otimes f^*_{\eta}(\uscr)).
    \end{tikzcd}
\end{equation*}
By \cite[Proposition 2.3.41]{ayoub-thesis-1}\footnote{In \cite{ayoub-thesis-1}, the author assumes varieties to be quasi-projective and use the factorization $f = h \circ i$ with $h$ smooth and $i$ closed; with varieties, we use the factorization $f = p \circ j$ with $p$ proper and $j$ open and the proof remains the same.}, there exists a commutative diagram 
\begin{equation*}
    \begin{tikzcd}[sep=large]
           (t^-)^*(e^-)^!(M) \otimes (t^-)^*(e^-)^*(C) \arrow[r,equal] \arrow[rrd,phantom,"(\text{C0})"] \arrow[d,color=blue] &  (t^-)^*((e^-)^!(M) \otimes (e^-)^*(C)) \arrow[r] & (s^-)^*(e^-)^!(M \otimes C) \arrow[d,color=blue] \\ 
            (t^+)^!(e^+)^*(M) \otimes (t^+)^*(e^+)^*(C) \arrow[r] &  (t^+)^!((e^-)^*(M) \otimes (e^-)^*(C)) \arrow[r,equal] & (t^+)^!(e^+)^*(M \otimes C).
    \end{tikzcd}
\end{equation*}
Apply $u^*= u^!$ to $(\text{C0})$ we obtain 
\begin{equation*}
    \begin{tikzcd}[sep=large]
           (s^-)^*(e^-)^!(M) \otimes (s^-)^*(e^-)^*(C) \arrow[r,equal] \arrow[rd,phantom,"(\text{C1})"] \arrow[d,color=blue]& (s^-)^*(e^-)^!(M \otimes C) \arrow[d,color=blue] \\ 
            (s^+)^!(e^+)^*(M) \otimes (s^+)^*(e^+)^*(C) \arrow[r] &   (s^+)^!(e^+)^*(M \otimes C).
    \end{tikzcd}
\end{equation*}
Moreover, the proof of \cite[Corollaire 2.1.91]{ayoub-thesis-1} shows that there exists a commutative diagram 
\begin{equation*}
    \begin{tikzcd}[sep=large]
          A \otimes B \arrow[r] \arrow[d] \arrow[rd,phantom,"(\text{C2})"] & (s^-)_*(s^-)^*(A \otimes B) \arrow[d,equal] \\ 
          (s^-)_*(s^-)^*(A) \otimes (s^-)_*(s^-)^*(B) \arrow[r,"\simeq"] & (s^-)_*(s^-)^*(A \otimes B),
    \end{tikzcd}
\end{equation*}
where $\id \longrightarrow (s^-)_*(s^-)^*$ indicates the unit morphism. We apply $(\pi^-)_*$ to (C2) and obtain 
\begin{equation*}
    \begin{tikzcd}[sep=large]
         (\pi^-)_*(A \otimes B) \arrow[r] \arrow[d]  & (\pi^-)_*(s^-)_*(s^-)^*(A \otimes B) \arrow[d,equal] \\ 
          (\pi^-)_*(s^-)_*(s^-)^*(A) \otimes (\pi^-)_*(s^-)_*(s^-)^*(B) \arrow[r] & (\pi^-)_*(s^-)_*(s^-)^*(A \otimes B)
    \end{tikzcd}
\end{equation*}
Now note that $\pi^- \circ s^- = \id$ and there exists a natural transformation $(\pi^-)_*(A) \otimes (\pi^-)_*(B) \longrightarrow (\pi^-)_*(A \otimes B)$. Therefore, we get  commutative diagram 
\begin{equation*}
    \begin{tikzcd}[sep=large]
        (\pi^-)_*(A) \otimes (\pi^-)_*(B) \arrow[r] \arrow[d,color=blue] \arrow[rd,phantom,"\text{(C3)}"] & (\pi^-)_*(A \otimes B) \arrow[d,color=blue] \\
          (s^-)^*(A) \otimes (s^-)^*(B)  \arrow[r,equal] & (s^-)^*(A \otimes B).
    \end{tikzcd}
\end{equation*} In particular, we obtain
\begin{equation*}
    \begin{tikzcd}[sep=large]
        (\pi^-)_* (e^-)^!(A) \otimes (\pi^-)_* (e^-)^*(C) \arrow[r] \arrow[rd,phantom,"\text{(C3)}"] \arrow[d,color=blue] & (\pi^-)_*( (e^-)^!(A) \otimes (e^-)^*(C)) \arrow[d,color=blue] \arrow[r] \arrow[dr,phantom,"\text{interchange law}"] & (\pi^-)_*(e^-)^!(A \otimes C) \arrow[d,color=blue] \\
          (s^-)^* (e^-)^!(A) \otimes (s^-)^* (e^-)^*(C) \arrow[r,equal] & (s^-)^*( (e^-)^!(A) \otimes  (e^-)^*(C)) \arrow[r,"\simeq"] & (s^-)^*(e^-)^!(A \otimes C).
    \end{tikzcd}
\end{equation*}
by making a substitution $A \longmapsto (e^-)^!(A)$ and $B \longmapsto (e^-)^*(C)$ and using the interchange law. By patching this one with the previous one, we get 
\begin{equation*}
    \begin{tikzcd}[sep=large]
        (\pi^-)_* (e^-)^!(A) \otimes (\pi^-)_* (e^-)^*(C) \arrow[r] \arrow[d,color=blue]  & (\pi^-)_*(e^-)^!(A \otimes C) \arrow[d,color=blue] \\
          (s^-)^* (e^-)^!(A) \otimes (s^-)^* (e^-)^*(C) \arrow[d,color=blue] \arrow[dr,phantom,"\text{(C1)}"] \arrow[r] & (s^-)^*(e^-)^!(A \otimes C) \arrow[d,color=blue] \\ 
           (s^+)^!(e^+)^*(A)\otimes (s^-)^* (e^-)^*(C) \arrow[r]   & (s^+)^!(e^+)^*(A \otimes C) .
    \end{tikzcd}
\end{equation*}
By setting $C = f^*(\uscr)$ and using $f \circ e^- = f^0 \circ \pi^-$ (since the $\gm_{m,k}$-action on $\mathbb{A}_k^1$ is trivial), we obtain 
\begin{equation*}
    \begin{tikzcd}[sep=large]
        (\pi^-)_* (e^-)^!(A) \otimes (\pi^-)_* (\pi^-)^*(f^0)^*(\uscr) \arrow[r] \arrow[d,color=blue]  & (\pi^-)_*(e^-)^!(A \otimes f^*(\uscr)) \arrow[d,color=blue] \\
          (s^+)^!(e^+)^*(A)\otimes (f^0)^*(\uscr)\arrow[r]  & (s^+)^!(e^+)^*(A \otimes f^*(\uscr)).
    \end{tikzcd}
\end{equation*}
By the projection formula, we get 
\begin{equation*}
    \begin{tikzcd}[sep=large]
    (\pi^-)_*(e^-)^!(A) \otimes (f^0)^*(\uscr) \arrow[d,"\text{unit}",swap,color=blue] \arrow[dr] & \\ 
        (\pi^-)_* (e^-)^!(A) \otimes (\pi^-)_* (\pi^-)^*(f^0)^*(\uscr) \arrow[r] \arrow[d,color=blue]  & (\pi^-)_*(e^-)^!(A \otimes f^*(\uscr)) \arrow[d,color=blue] \\
          (s^+)^!(e^+)^*(A)\otimes (f^0)^*(\uscr)\arrow[r] \arrow[dr,phantom,"\text{interchange law}"] \arrow[d,color=blue] & (s^+)^!(e^+)^*(A \otimes f^*(\uscr)) \arrow[d,color=blue]  \\ 
           (\pi^+)_! (e^+)^*(A) \otimes (f^0)^*(\uscr) \arrow[r]  \arrow[rd,"\text{projection formula}",swap]  & (\pi^+)_!(e^+)^*(A \otimes f^*(\uscr)) \arrow[d,equal] \\ 
           & (\pi^+)_!((e^+)^*(A) \otimes (\pi^+)^*(f^0)^*(\uscr))
    \end{tikzcd}
\end{equation*}
as desired.
\end{proof}

\subsection{The Integral Identity} As in \cite{bang-2024}, let us deduce the motivic integral identity as a consequence of the commutativity between motivic nearby functors and Braden's transformations. The following lemma is easy.
\begin{lem}
Let $Y$ be a $k$-variety. Let $\gm_{m,k}$ act on $X=\mathbb{A}^{d_1} \times_k \mathbb{A}^{d_2} \times_k Y$ by (on points) 
\begin{equation*}
            \lambda(\mathbf{a}_1,\mathbf{a}_2,x) = (\lambda^{>0}\mathbf{a}_1,\lambda^{<0}\mathbf{a}_2,x)  \ \forall \ \lambda \neq 0
        \end{equation*}
then one has that $X^+ = \mathbb{A}_k^{d_1} \times_k Y, X^- = \mathbb{A}_k^{d_2} \times_k Y, X^0 = Y$ and $\pi^{\pm}$ are projections and $e^{\pm}$ are closed immersions given by zero sections. 
\end{lem}
\begin{proof}
        Let us compute $X^+$ for instance. The other cases are similar. On points, one has that $X^+(T)$ is the set of morphisms $f = (f_1,f_2,f_3) \colon \mathbb{A}^1_k  \longrightarrow \mathbb{A}^{d_1}_k \times_k \mathbb{A}^{d_2}_k \times_k X$ such that $f(\lambda x) = (f_1(\lambda x),f_2(\lambda x),f_3(\lambda x)) = (\lambda^{>0} f_1(x), \lambda^{<0} f_2(x), f_3(x))$. Thus, $f_2(\lambda x) =\lambda^{<0}f_2(x) $, which does not happen for algebraic functions unless $f_2=0$. Meanwhile, $f_3$ must be constant and $f_1$ is of the form $(c_1 a_1^{\bullet},...,c_{d_1}a_{d_2}^{\bullet})$ with those $\bullet$ fixed. The map $X^+ \longrightarrow \mathbb{A}_k^{d_1} \times_k \xfrak$ given by $f \longmapsto (c_1,....,c_{d_1},f_3)$ determines the isomorphism of functors of points. 
    \end{proof}
\begin{theorem}
Let $X,Y$ be as in the previous lemma. Let $Y$ be embedded in $\mathbb{A}_k^{d_1} \times_k \mathbb{A}_k^{d_2} \times_k Y$ by zero sections, i.e. 
\begin{equation*}
    Y \longhookrightarrow \mathbb{A}_k^{d_1} \times_S \mathbb{A}_k^{d_2} \times_k Y \ \ \ \ \      y  \longmapsto (0,0,y).
\end{equation*}
Let $f \colon \mathbb{A}_k^{d_1} \times_k \mathbb{A}_k^{d_2} \times_k Y \longrightarrow \mathbb{A}_k^1$ be a $\mathbb{G}_{m,k}$-equivariant morphism, where $\mathbb{A}_k^1$ is endowed with the trivial $\mathbb{G}_{m,k}$-action, i.e., on points,
\begin{equation*}
    f(\mathbf{a}_1,\mathbf{a_2},x) = f(\lambda^{>0}\mathbf{a}_1,\lambda^{<0}\mathbf{a}_2,x).
\end{equation*} 
Then there is an isomorphism of motives
 \begin{equation*}
        (\pi^+_{\sigma})_!(e^+_{\sigma})^*\Psi_f(\mathds{1}_{X_{\eta}}) = \int_Y\big(\Psi_f(\mathds{1}_{X_{\eta}}) \big)_{\mid \mathbb{A}^{d_1}_k \times Y} \simeq \mathds{1}_{X_{\sigma}}(d_1)[2d_1] \otimes  \Psi_{f_{\mid Y}}(\mathds{1}_{Y_{\eta}})
  \end{equation*}
\end{theorem}

\begin{proof}
The proof remains the same as the proof of \cite[Theorem 4.3]{bang-2024} (see also, \cite[Theorem 3.3.3]{bang-thesis}).
\end{proof}

\section{Appendix: Beilinson's theorem for perverse sheaves of geometric origin}

In \cite{beilinson-1987-2}, Beilinson proves that the derived category of perverse sheaves is equivalent to the constructible derived category. One may ask a natural question that can we prove the same theorem for perverse sheaves of geometric origin. Such theorem is implicitly used throughout this paper and therefore we write this appendix with the purpose of providing a complete proof for that theorem. Let us meticulously follow the proof written in \cite[Chapter 4]{achar-book}. 
\subsection{On Beilinson's realization functor} Let $\tcal$ be a triangulated category equipped with a bounded $t$-structure whose heart is $\acal$. In case $\tcal$ admits a \textit{filtered version}, then in \cite{beilinson-1987-2}, Beilinson constructs a realization functor 
\begin{equation*}
    \operatorname{real} \colon \dcat^b(\acal) \longrightarrow \tcal.
\end{equation*}
In practice, we do not have to care about what it means to have a filtered version because all full triangulated subcategory of the bounded derived category of an abelian category satisfies such condition. In particular, there are realization functors
\begin{align*}
    \operatorname{real} \colon \dcat^b(\perv(X)) & \longrightarrow \dcatct^b(X) \\ 
    \operatorname{real} \colon \dcat^b(\perv_{\geo}(X)) & \longrightarrow \dcat^b_{\geo}(X).
\end{align*}
A necessary and sufficient condition for $\operatorname{real}$ to be an equivalence is that any morphism $X \longrightarrow Y[n]$ (with $X,Y \in \acal$, $n >0$) is effaceable, namely, there exists an injective morphism $Y \longhookrightarrow I$ such that $X \longrightarrow Y[n] \longrightarrow I[n]$ is zero. For a proof, the reader can consult \cite[Lemma 1.4]{beilinson-1987-2}\cite[Corollary A.7.19]{achar-book}.
\subsection{Perverse sheaves of geometric origin and how to glue them} Let $k \subset \mathbb{C}$ be a subfield and $X/k$ be a variety. Let $X^{\an}$ be its $\mathbb{C}$-points equipped with the complex topology. Let $\dcatct^b(X)$ be the bounded category with algebraically constructible cohomologies. Let $\dcat^b_{\geo}(X) \subset \dcatct^b(X)$ be a smallest sub-$\infty$-category spanned by the thick triangulated subcategory (at the level of homotopy category) containing all complexes of the form $p_*(\mathds{1}_Z)$ with $p \colon Z \longrightarrow X$ a proper morphism. 
\begin{prop} \label{prop: perversity of operations on geometric perverse sheaves}
    The collection $\dcat_{\geo}^b(X)$ is stable under six operations. Moreover, the six operations satisfy the same perversity as they do for ordinary perverse sheaves. 
\end{prop}

\begin{proof}
    Let $f \colon X \longrightarrow Y$ be a morphism of $k$-varieties. The stability of $f^*$ follow from the proper base change theorem. Now it is clear that the Verdier duality $\mathbb{D}_X$ induces an involution on $\dcat^b_{\geo}(X)$ since $\mathbb{D}_X p_*(\mathds{1}_Z)= p_* \mathbb{D}_Z(\mathds{1}_Z) = p_*(\mathds{1}_Z)$. Thus, $f^! = \mathbb{D}_X f^*\mathbb{D}_Y$ sends $\dcat^b_{\geo}(Y)$ to $\dcat^b_{\geo}(X)$. A similar argument proves that $f_*$ preserves complexes of geometric origin if and only if $f_!$ does. Regard the case $f_!$, if $f$ is proper, then it becomes trivial. In general, we have to prove that $(f \circ p)_!(\mathds{1}_Z)$ is of geometric origin. Using the Nagata compactification theorem and the case of $f$ being proper, it suffices to prove that $j_!(\mathds{1})$ is of geometric with $j$ an open immersion. By the localization sequence, we know that $j_!(\mathds{1}) = \operatorname{Fib}(\mathds{1} \longrightarrow i_*\mathds{1})$ with $i$ being the closed complement (endowed with the reduced structure) and we win. The stability of $\otimes$ follows from the Kunneth formula. The stability of $\underline{\Hom}$ follows from the for mula (with $p,q$ proper) $\underline{\Hom}(p_*(\mathds{1}),q_*(\mathds{1})) \simeq p_*\underline{\Hom}(\mathds{1},p^!q_*(\mathds{1})) \simeq p_*p^!q_*(\mathds{1})$ and the stabilities of functors $f_*,f_!$ showed before. 
\end{proof}

\begin{prop} \label{prop: restrictions of nearby functors on geometric complexes}
    Let $f \colon X \longrightarrow \mathbb{A}_k^1$ be a morphism of $k$-varieties. The collection $\dcat_{\geo}^b(X)$ is stable under the nearby functor $\Psi_f^{\an}$, the unipotent nearby functor $\Upsilon_f^{\an}$, the vanishing cycles functor $\Phi_f^{\an}$ and all these functors are perverse $t$-exact.
\end{prop}

\begin{proof}
     If these functors are well-restricted, then their perverse $t$-exactnesses follow from the fact that $\dcat^b_{\geo}(X) \longrightarrow \dcatct^b(X)$ is conservative and perverse $t$-exact. The functor $\Psi_f^{\an}$ preserves complexes of geometric origin since by \cite[Proposition 4.8]{ayoub-2010}, we can compute it via the Betti realization and the motivic nearby functor and conclude by the fact that \'etale motives are of geometric origin. The functor $\Upsilon_f^{\an}$ is a direct factor of $\Psi_f^{\an}$ and it is also well-defined. The same holds for $\Phi_f^{\an} = \operatorname{Cofib}(i^* \longrightarrow \Psi_f^{\an}j^*)$. 
\end{proof}

\begin{prop}
     Let $f \colon X \longrightarrow \mathbb{A}_k^1$ be a morphism of $k$-varieties. The collection $\dcat_{\geo}^b(X)$ is stable under the Beilinson's unipotent nearby functor $\Psi_f^{\an}$, the Beilinson's vanishing cycles functor $\Upsilon_f^{\an}$, the maximal extension functor $\Xi_f^{\an}$. In particular, one can glue perverse sheaves of geometric origin. 
\end{prop}

\begin{proof}
    There exists a short exact sequence $0 \longrightarrow i_*\Psi_f^{\an} \longrightarrow \Xi_f \longrightarrow j_* \longrightarrow 0$ on $\perv(X)$ and hence it lifts to a triangle on $\dcatct^b(X)$ (modulo the usual - not the geometric one we are going to prove - Beilinson's equivalence \cite{beilinson-1987-2}). 
\end{proof}
Since one can glue perverse sheaves of geometric origin, the following is an obvious consequence. 
\begin{cor} \label{cor: perverse shv of geometric origin supported on a closed subscheme}
    Let $f \colon X \longrightarrow \mathbb{A}_k^1$ be a morphism of $k$-varieties. Let $i \colon X_{\sigma} \longrightarrow X$ be the closed immersion of the special fiber, then the exact functor
    \begin{equation*}
        i_* \colon \perv_{\geo}(X_{\sigma}) \longrightarrow \perv_{\geo,X_{\sigma}}(X)
    \end{equation*}
    is an equivalence of categories. 
\end{cor}
\subsection{Local systems of geometric origin}
Let $X/k$ be a smooth variety and we set $\localsystem_{\geo}(X) = \perv_{\geo}(X) \cap \localsystem(X)$ to be the Serre subcategory of $\perv_{\geo}(X)$ containing local systems of geometric origin. Let 
$\dcat^b_{\loc_{\geo}}(X) \subset \dcat^b_{\geo}(X)$ be the full subcategory whose objects are complexes having (constructible) cohomologies being shifted local systems of geometric origin. Alternatively, we can define $\dcat^b_{\log_{\geo}}(X)$ to be the subcategory whose perverse cohomologies are local systems of geometric origin. It is clear that the category $\dcat^b_{\loc_{\geo}}(X)$ admits a perverse $t$-structure whose heart is $\localsystem_{\geo}(X)[\dim(X)]$ and that the natural embedding $\dcat^b_{\loc_{\geo}}(X) \longrightarrow \dcat^b_{\geo}(X)$ is perverse $t$-exact whose restriction to hearts is the natural embedding $\localsystem_{\geo}(X)[\dim(X)] \longhookrightarrow \perv_{\geo}(X)$.
\begin{prop} \label{prop: find an open subscheme with right effaceable property}
    Let $X$ be a smooth, connected $k$-variety, then there exists an affine, open subscheme $U \subset X$ such that 
    \begin{equation*}
        \Hom_{\dcat_{\loc_{\geo}}^b(U)}(\mathds{1}_U[\dim(U)],-) \colon \dcat_{\loc_{\geo}}^b(U) \longrightarrow \operatorname{Vect}_{\mathbb{Q}}
    \end{equation*}
    is right effaceable. 
\end{prop}
\begin{proof}
We prove the result by induction on the dimension of $X$. If $\dim(X)=0$, there is nothing to prove. Let us assume that $\dim(X) \geq 1$. At this point, the proof goes the same way as the proof of \cite[Proposition 4.5.4]{achar-book} with some slight modifications. If $X$ is an affine open subscheme of $\mathbb{A}^1$. By Artin's vanishing theorem, we know that $\Hom_{\loc_{\geo}}(\mathds{1}_X,L[n])=0$ unless $n=0,1$. We pick a $\mathbb{Q}$-basis $e_1,...,e_m$ of $\Hom_{\loc_{\geo}}(\mathds{1}_X,L[1])$ and comprise into a morphism $e=(e_1,...,e_m) \colon \mathds{1}_X^{\oplus m} \longrightarrow L[1]$. We complete this morphism to a triangle and rotate to get
    \begin{equation*}
        L[\dim(X)] \longrightarrow F[\dim(X)] \longrightarrow (\mathds{1}_X[\dim(X)])^{\oplus m} \overset{e}{\longrightarrow} +1.
    \end{equation*}
    Since $\localsystem_{\geo}(X)[\dim(X)] \subset \perv_{\geo}(X)$ is a Serre subcategory, we see that $F \in \localsystem_{\geo}(X)$ and $L \longrightarrow F$ is injective. From now on, one can copy \textit{Step 2,3} of the proof of \cite[Proposition 4.5.4]{achar-book}. The only thing one needs to take care is a version of \cite[Lemma 4.5.3]{achar-book} for $\dcat^b_{\loc_{\geo}}(-)$ instead of $\dcat^b_{\loc}(-)$. In the proof of \cite[Lemma 4.5.3]{achar-book}, the only "non-functorial" point is \cite[Theorem 1.9.5]{achar-book}. Thus, we have to show that if $f \colon X \longrightarrow Y$ is a locally trivial fibration, then $f_*(L) \in D^+_{\loc_{\geo}}(Y)$ whenver $L \in \localsystem_{\geo}(X)$. By \cite[Theorem 1.9.5]{achar-book}, we already know that $f_*(L) \in D^b_{\loc}(Y)$, but since $\dcat^b_{\loc_{\geo}}(Y) = \dcat^b_{\geo}(Y) \cap \dcat^b_{\loc}(Y)$, we are done. 
\end{proof}
The following result is essentially due to Beilinson (see \cite{beilinson-1987-2}) and is rewritten (with slight generalizations) in \cite[Theorem 1.5.9]{ayoub-2025}. By \cite{beilinson-1987-2} again, we have a realization functor $\operatorname{real} \colon \dcat^b(\localsystem_{\geo}(X)) \longrightarrow \dcat^b_{\loc_{\geo}}(X)$ provided that $X$ is smooth. 
\begin{prop} \label{prop: Beilinson theorem for local systems of geometric origin}
    Let $X/k$ be a variety, then there exists a smooth, connected, affine open subscheme $U \subset X$ such that the Beilinson's realization functor
    \begin{equation*}
        \operatorname{real} \colon \dcat^b(\localsystem_{\geo}(U)) \longrightarrow \dcat^b_{\loc_{\geo}}(U)
    \end{equation*}
    is an equivalence of categories. 
\end{prop}

\begin{proof}
   The proof remains the same as the proof of \cite[Theorem 4.5.5]{achar-book} with the follows to be kept in mind: we use proposition \ref{prop: find an open subscheme with right effaceable property} as a replacement of \cite[Proposition 4.5.4]{achar-book} and to conclude, we need a version of \cite[Lemma 4.5.2]{achar-book} for local systems of geometric origin. However, the proof of \cite[Lemma 4.5.2]{achar-book} is of purely six-functor formalism nature so we win since we know that six operations preserve perverse sheaves of geometric origin.
\end{proof}
\subsection{Beilinson's theorem for perverse sheaves of geometric origin} Now we are fully equipped to prove the geometric version of Beilinson's equivalence. 
\begin{lem} \label{lem: adjunctions on perverse sheaves of geometric}
    Let $j \colon U \longrightarrow X$ be an open immersion of $k$-varieties with $U$ affine. The canonical maps
    \begin{align*}
        \Hom_{\dcat^b(\perv_{\geo}(U))}(j^*M,N[k]) \longrightarrow  \Hom_{\dcat^b(\perv_{\geo}(X))}(M,j_*N[k])  \\ 
         \Hom_{\dcat^b(\perv_{\geo}(X))}(j_!M,N[k]) \longrightarrow  \Hom_{\dcat^b(\perv_{\geo}(U))}(M,j^*N[k])
    \end{align*}
    are isomorphisms.
\end{lem}
\begin{proof}
    This simply follows from the triangle identities and that $j_!,j_*,j^*$ are $t$-exact on $D^b(\perv_{\geo}(U))$.
\end{proof}
\begin{lem} \label{lem: full faithfulness of closed immersions on perverse sheaves of geometric origin}
    Let $i \colon Z \longhookrightarrow X$ be a closed immersion. The functor 
    \begin{equation*}
        i_* \colon \dcat^b(\perv_{\geo}(X_{\sigma})) \longhookrightarrow \dcat^b(\perv_{\geo}(X))
    \end{equation*}
    is fully faithful. Its essential image, denoted by $\dcat^b_{Z}(\perv_{\geo}(X))$, is the kernel of 
    \begin{equation*}
        j^* \colon \dcat^b(\perv_{\geo}(U)) \longrightarrow \dcat^b(\perv_{\geo}(X))
    \end{equation*}
    where $j \colon U \longhookrightarrow X$ is the open complement. 
\end{lem}
\begin{proof} One can copy the proof of \cite[Theorem 4.1]{florian+morel-2019} (see also \cite[Proposition 4.5.7]{achar-book}). 
\end{proof} 
\begin{theorem} \label{thm: Beilinson's theorem for geometric complexes}
    Let $X$ be a $k$-variety, then the Beilinson's realization functor
    \begin{equation*}
         \operatorname{real} \colon \dcat^b(\perv_{\geo}(X)) \longrightarrow \dcat^b_{\geo}(X)
    \end{equation*}
    is an equivalence of categories. 
\end{theorem}
\begin{proof}
 Given proposition \ref{prop: Beilinson theorem for local systems of geometric origin}, lemmas \ref{lem: adjunctions on perverse sheaves of geometric} and \ref{lem: full faithfulness of closed immersions on perverse sheaves of geometric origin}, one can copy the proof of \cite[Theorem 4.5.9]{achar-book}. The only point that one needs to care is to make sure one can do induction on the number of composition factors of a perverse sheave of geometric origin. This follows from the fact that $\perv_{\geo}(X) \subset \perv(X)$ is a Serre subcategory, which implies that $\perv_{\geo}(X)$ is both noetherian and artinian. 
\end{proof}
\bibliographystyle{alpha}
\bibliography{ref}

\end{document}